\documentclass[11pt]{article}
\usepackage{graphicx}
\usepackage{amscd, amssymb}
\usepackage{enumerate}
\usepackage{hyperref}
\usepackage{lscape}
\newenvironment{eq}{\begin{equation}}{\end{equation}}
\newenvironment{proof}{{\bf Proof}:}{\vskip 5mm }

\newtheorem{proposition}{Proposition}[subsection]
\newtheorem{lemma}[proposition]{Lemma}
\newtheorem{definition}[proposition]{Definition}
\newtheorem{theorem}[proposition]{Theorem}

\newtheorem{example}[proposition]{Example}
\newtheorem{remark}[proposition]{Remark}

\newtheorem{problem}[proposition]{Problem}
\newtheorem{construction}[proposition]{Construction}
\newcommand{\llabel}[1]{\label{#1}}
\newcommand{\comment}[1]{}
\newcommand{\sr}{\rightarrow}

\newcommand{\nn}{{\bf N\rm}}
\newcommand{\nat}{\nn}

\newcommand{\wt}{\widetilde}

\newcommand{\wh}{\widehat}

\newcommand{\mbind}{\rho}
\newcommand{\hc}{\circ_{T}}

{\vskip
3mm}

\newcommand{\spc}{{\,\,\,\,\,\,\,}}

\newcommand{\RR}{{\bf RR}}
\newcommand{\LM}{{\bf LM}}

\begin{document}
\parskip = 2mm
\begin{center}
{\bf\Large C-system of a module over a $Jf$-relative monad\footnote{\em 2000 Mathematical Subject Classification: 
18D99, % category theory and homological algebra, categories with structures, none of the above, but in this section
18C50  %	Categorical semantics of formal languages 
}}

%{\keywords Contextual categories, relative monads, modules over monads.}

\vspace{3mm}

{\large\bf Vladimir Voevodsky}\footnote{School of Mathematics, Institute for Advanced Study,
Princeton NJ, USA. e-mail: vladimir@ias.edu}
\vspace {3mm}

{\large\bf September 2015}  
\end{center}

\begin{abstract}
This is the second paper in a series started in \cite{Csubsystems}. Let $F$ be the category with the set of objects $\nat$ and morphisms being the functions between the standard finite sets of the corresponding cardinalities. Let $Jf:F\sr Sets$ be the obvious functor from this category to the category of sets. In this paper we construct, for any relative monad $\RR$ on $Jf$ and a left module $\LM$ over $\RR$, a C-system $C(\RR,\LM)$ and explicitly compute the action of the four B-system operations on its B-sets. 
In the following paper it is used to provide a rigorous mathematical approach to the construction of the C-systems underlying the term models of a wide class of dependent type theories. 
\end{abstract}

%(2014.09.27) Make a note about the functoriality of C(\RR,\LM) on the "large module category" %of Hirschowitz-Maggesi. 
%Change the name of the monad from M to R. Also \mu, \eta for the monad structure and 
%\rho for the module structure.

%$$\mathfrak{S}$$

%??? Make the dependence on the choice of a universe explicit? Then need to understand what is required from a set $UU$ or a type universe $UU$ in order to be able to realize the results of this paper. Can one consider LM in Prop? Is it relevant for the formalization of predicate logic?

\tableofcontents

\subsection{Introduction}

The first few steps in all approaches to the semantics of dependent type theories remain insufficiently understood. The constructions which have been worked out in detail in the case of a few particular type systems by dedicated authors are being extended to the wide variety of type systems under consideration today by analogy. This is not acceptable in mathematics. Instead we should be able to obtain the required results for new type systems by {\em specialization} of general theorems and constructions formulated for abstract objects the instances of which combine together to produce a given type system. 

%???!!## Explain through reference to Martin Hofmann Section 2.3 about pre-syntax. The word "associative" does not even appear in that paper. Sections 2.3, 2.4 The "construction" of term model (which is what the present paper is about) in 3.1 lacks proofs and even precise statements entirely. 

%Also a reference to Jacob where he says that proving associativity is non-trivial. 

%The Ty/Tm definition of CwF appears already in Hofmann (3.1). 

%Mention that our description of general sub-quotients allows to use our results for the study of the semantics of type systems with context-dependent computation.

An approach that follows this general philosophy was outlined in \cite{CMUtalk}. In this approach the connection between the type theories, which belong to the concrete world of logic and programming, and abstract mathematical concepts such as sets or homotopy types is constructed through the intermediary of C-systems. 

C-systems were introduced in \cite{Cartmell0} (see also \cite{Cartmell1}) under the name ``contextual categories''. A modified axiomatics of C-systems and the construction of new C-systems as sub-objects and regular quotients of the existing ones in a way convenient for use in type-theoretic applications are considered in \cite{Csubsystems}. A C-system equipped with additional operations corresponding to the inference rules of a type theory is called a model or a C-system model of these rules or of this type theory. There are other classes of objects on which one can define operations corresponding to inference rules of type theories most importantly categories with families or CwFs. They lead to other classes of models.  

In the approach of \cite{CMUtalk}, in order to provide a mathematical representation (semantics) for a type theory one constructs two C-systems. One C-system, which we will call the proximate or term C-system of a type theory, is constructed from formulas of the type theory using, in particular, the main  construction of the present paper.  The second C-system is constructed from the category of abstract mathematical objects using the results of \cite{Cfromauniverse}. Both C-systems are then equipped with additional operations corresponding to the inference rules of the type theory making them into models of type theory.  The model whose underlying C-system is the term C-system is called the term model. 

A crucial component of this approach is the expected result that for a particular class of the inference rules the term model is an initial object in the category of models. This is known as the Initiality Conjecture. In the case of the pure Calculus of Constructions with a ``decorated''  application operation this conjecture was proved in 1988 by Thomas Streicher \cite{Streicher}. The problem of finding an appropriate formulation of the general version of the conjecture and of proving this general version will be the subject of future work. 

For such inference rules, then, there is a unique homomorphism from the term C-system to the abstract C-system that is compatible with the corresponding systems of operations. Such homomorphisms are called representations of the type theory. More generally, any functor from the category underlying the term C-system of the type theory to another category may be called a representation of the type theory in that category. Since objects and morphisms of term models are built from formulas of the type theory and objects and morphisms of abstract C-systems are built from mathematical objects such as sets or homotopy types and the corresponding functions, such representations provide a mathematical meaning to formulas of type theory. 

The existence of these homomorphisms in the particular case of the ``standard univalent models'' of  Martin-L\"{o}f type theories and of the Calculus of Inductive Constructions (CIC) provides the only known justification for the use of the proof assistants such as Coq for the formalization of mathematics in the univalent style (see \cite{UniMath}, \cite{UniMath2015}). 

Only if we know that the initiality result holds for a given type theory can we claim that a model defines a representation. A similar problem also arises in the predicate logic but there, since one considers only one fixed system of syntax and inference rules, it can and had been solved once without the development of a general theory. The term models for a class of type theories can be obtained by considering slices of the term model of the type theory called Logical Framework (LF), but unfortunately it is unclear how to extend this approach to type theories that have more substitutional (definitional) equalities than LF itself.

A construction of a model for the version of the Martin-L\"{o}f type theory that is used in the UniMath library (\cite{UniMath},\cite{UniMath2015}) is sketched in \cite{KLV1}. At the time when that paper was written it was unfortunately assumed that a proof of the initiality result can be found in the existing body of work on type theory which is reflected  in \cite[Theorem 1.2.9]{KLV1} (cf. also \cite[Example 1.2.3]{KLV1} that claims as obvious everything that is done in both the present paper and in \cite{Csubsystems}).  Since then it became clear that this is not the case and that a mathematical theory leading to the initiality theorem and providing a proof of such a theorem is lacking and needs to be developed. 

As the criteria for what constitutes an acceptable proof were becoming more clear as a result of continuing work on formalization, it also became clear that more detailed and general proofs need to be given to many of the theorems of \cite{KLV1} that are related to the model itself. For the two of the several main groups of inference rules of current type theories it is done in \cite{fromunivwithPi} and \cite{fromunivwithpaths}. Other groups of inference rules will be considered in further papers of that series. 

In this paper we describe a purely algebraic construction that defines a C-system $C(\RR,\LM)$ starting with a pair $(\RR,\LM)$ where $\RR$ is a relative monad on the functor $Jf:F\sr Sets$ (see below) and $\LM$ is a (left) module over this monad. 

This construction provides a step in the path from the description of a type theory by a collection of inference rules, as is customary in the type theory papers, to the term model of this type theory as a C-system equipped with a system of operations corresponding to these rules. 

On this path one starts by defining from the inference rules a two-sorted binding signature that describes the raw syntax of type and element constructors of the type theory. Then one defines from this two-sorted binding signature a pair $(\RR,\LM)$ and, applying the construction of this paper, obtains the C-system $C(\RR,\LM)$ of the raw syntax of the theory. 

Such C-systems have not been considered previously probably because from the perspective of logic they are hard to interpret. However they provide a very convenient stepping stone to more complex term C-systems of type theories. 

We defer the detailed descriptions both of the step preceding the one described here and of the one following it to future papers. In the remaining part of the introduction we describe the content of the paper without further references to type theory. 

We start the paper with two sections where we introduce some constructions applicable to general C-systems. 

On the sets of objects of any C-system one can consider the partial ordering defined by the condition that $X\le Y$ if and only if $l(X)\le l(Y)$ and $X=ft^{l(Y)-l(X)}(Y)$. In the first section we re-introduce some of the objects and constructions defined in \cite{Csubsystems} using the length function using this partial ordering instead. This allows to avoid the use of natural numbers in some of the arguments that significantly simplifies the proofs. 

In the second section we construct for any C-system $CC$ and a presheaf $F$ on the category underlying $CC$ a new C-system $CC[F]$ that we call the $F$-extension of $CC$. The C-systems of this form remind in some way the affine spaces over schemes in algebraic geometry. While the geometry of affine spaces in itself is not very interesting their sub-spaces encompass all affine algebraic varieties of finite type . Similarly, while the C-systems $CC[F]$ look to be not very different from $CC$ their sub-systems and more generally regular sub-quotients, even in the case of the simplest C-systems $CC=C(\RR)$ corresponding to Lawvere theories (see Section \ref{CRR}), include all of the term C-systems of type theories.  

Regular sub-quotients of any C-system $CC$ are classified by quadruples $(B,\wt{B}, \sim,\simeq)$ of the following form. 

Let $\wt{Ob}(CC)$ be the set of sections of the p-morphisms of $CC$, i.e., the subset in $Mor(CC)$ that consists of morphisms $s$ such that $dom(s)=ft(codom(f))$ and $s\circ p_{codom(f)}=Id_{dom(f)}$.  The sets $Ob(CC)$ and $\wt{Ob}(CC)$ are called the B-sets of a C-system and can also be denoted as $B(CC)$ and $\wt{B}(CC)$. 

The first two components $B$ and $\wt{B}$ of the quadruple are subsets in the sets $Ob(CC)$ and $\wt{Ob}(CC)$ respectively. The next two components are equivalence relations on $B$ and $\wt{B}$. To correspond to a regular sub-quotient the pair $(B,\wt{B})$ should be closed under the eight B-system operations on $(B(CC),\wt{B}(CC))$ and the equivalence relations of the pair $(\sim,\simeq)$ should be compatible with the restrictions of these eight operations to $(B,\wt{B})$ as well as to satisfy three additional simple conditions (see \cite[Proposition 5.4]{Csubsystems}) that involve the length function $l:B\sr\nat$ on $B$. 

Therefore, in order to be able to describe regular sub-quotients of a C-system one needs to know the B-sets of this C-system, the length function and the action of 
the eight B-system operations on these sets. Such a collection of data is called a pre-B-system (see \cite{Bsystems}). The main result of this paper is a detailed description of the pre-B-systems of the form $(B(CC[F]),\wt{B}(CC[F]))$ for a particular class of ``coefficient'' C-systems $CC$ (see below). 

In Section \ref{Jfrel} we first remind the notion of a relative monad on a functor $J:C\sr D$ that was introduced in \cite[Def.1, p. 299]{ACU} and considered in more detail in \cite{ACU2}. Then we focus our attention on relative monads over the functor $Jf$ that is defined as follows. 

For two sets $X$ and $Y$ let $Fun(X,Y)$ be the set of functions from $X$ to $Y$.  Let $stn(n)$ be the standard set with $n$ elements that we take to be the subset of $\nat$ that consists of numbers $<n$. Consider the category $F$ such that $Ob(F)=\nn$ and
$$Mor(F)=\cup_{m,n}Fun(stn(m),stn(n))$$
The functor $Jf$ is the obvious functor from $F$ to the category of sets. 

In \cite{LandJf} we constructed an equivalence between the category of $Jf$-relative monads and the category of Lawvere theories whose component functor from the relative monads to Lawvere theories is denoted $RML$. A key component of this equivalence is the construction of the Kleisli category $K(\RR)$ of a relative monad $\RR$ given in \cite{ACU2}.  Most of Section \ref{Jfrel} is occupied by simple computations in $K(\RR)$ for $Jf$-relative monads $\RR$.

In \cite{LandC} we constructed an isomorphism between the category of Lawvere theories and the category of l-bijective C-systems - the C-systems $CC$ where the length function $Ob(CC)\sr \nat$ is a bijection. In Section \ref{CRR} we consider the C-system $C(\RR)$ corresponding to the Lawvere theory $RML(\RR)$ defined by a $Jf$-relative monad $\RR$. The underlying category of this C-system is $K(\RR)^{op}$. The main result of this section is the description of the B-sets of $C(RR)$ and of the actions of the B-system operations on these sets.

In the final Section \ref{CRRLM} we apply the construction of Section \ref{Fext} to $C(\RR)$ taking into account that the functors $\LM:C(\RR)^{op}\sr Sets$ are the same as the functors $K(\RR)\sr Sets$ that are the same as the relative (left) modules over $Jf$. In (\ref{2016.01.21.eq3}) and Construction \ref{2015.08.22.constr1} we compute the B-sets $B(C(\RR,\LM))$ and $\wt{B}(C(\RR,\LM))$ and in Theorem \ref{2015.08.26.th2} the action of the B-system operations on these sets. 

In the next paper we will connect these computations to the conditions that the valid judgements of a type theory must satisfy in order for the term C-system of this type theory to be defined. 

\vspace{5mm}

Since this paper as well as other papers in the series on C-systems is expected to play a role in the mathematically rigorous construction of  the simplicial univalent representation of the UniMath language and the Calculus of Inductive Constructions and since such a construction itself can not rely on the univalent foundations the paper is written from the perspective of the Zermelo-Fraenkel formalism. 

\vspace{5mm}

The methods of the paper are fully constructive. We use neither the axiom of excluded middle nor the axiom of choice. The paper is written in the formalization-ready style and should be easily formalizable both in the UniMath and in the ZF. 

\vspace{5mm}

We use the diagrammatic order of composition, i.e., for morphisms $f:X\sr Y$ and $g:Y\sr Z$ we write their composition as $f\circ g$. 

We fix a universe $U$ without making precise what conditions on the set $U$ we require. It is clear that it is sufficient for all constructions of this paper to require $U$ to be a Grothendieck universe. However, it is likely that a much weaker set of conditions on $U$ is sufficient for our purposes. In all that follows we write $Sets$ instead of $Sets(U)$.  

This is one the papers extending the material which I started to work on in \cite{NTS}. I would like to thank the Institute Henri Poincare in Paris and the organizers of the ``Proofs'' trimester for their hospitality during the preparation of the first version of this paper. The work on this paper was facilitated by discussions with Benedikt Ahrens, Richard Garner and Egbert Rijke.  

%###??? Make a note that we do not use dependent types in the paper which is why we can not simply write Hom(\wh{n},\wh{m})=R(stn(n))^{stn(m)}

\subsection{Some general remarks on C-systems}
\llabel{onCsystems}
Recall that for a C-system $CC$, and object $\Gamma$ of $CC$ such that $l(\Gamma)\ge i$ we let $p_{\Gamma,i}$ denote the morphism $\Gamma\sr ft^i(\Gamma)$ defined inductively as
$$p_{\Gamma,0}=Id_{\Gamma}$$
$$p_{\Gamma,i+1}=p_{\Gamma}\circ p_{ft(\Gamma),i}$$
For $\Gamma'$ such that $l(\Gamma')\ge i$ and $f:\Gamma\sr ft^i(\Gamma')$ we let $f^*(\Gamma',i)$ and 
$$q(f,\Gamma',i):f^*(\Gamma',i)\sr \Gamma'$$ 
define a pair of an object and a morphism defined inductively as
\begin{eq}\llabel{2016.01.31.eq2}
f^*(\Gamma',0)=\Gamma\spc q(f,\Gamma',0)=f$$
$$f^*(\Gamma', i+1)=q(f,ft(\Gamma'),i)^*(\Gamma')\spc q(f,\Gamma',i+1)=q(q(f,ft(\Gamma'),i),\Gamma')
\end{eq}

For $\Gamma,\Gamma'$ in a C-system let us write $\Gamma\le \Gamma'$ if $l(\Gamma)\le l(\Gamma')$ and $\Gamma=ft^{l(\Gamma')-l(\Gamma)}(\Gamma')$. We will write $\Gamma<\Gamma'$ if $\Gamma\le \Gamma'$ and $l(\Gamma)<l(\Gamma')$. 

If $\Gamma'$ is over $\Gamma$ we will denote by $p_{\Gamma',\Gamma}$ the morphism
$$p_{\Gamma',l(\Gamma')-l(\Gamma)}:\Gamma'\sr \Gamma$$
If $\Gamma'$ and $\Gamma''$ are over $\Gamma$ then we have morphisms 
$$p_{\Gamma',\Gamma}:\Gamma'\sr \Gamma$$
$$p_{\Gamma'',\Gamma}:\Gamma''\sr\Gamma$$
and we say that a morphism $f:\Gamma'\sr \Gamma''$ is over $\Gamma$ if 
$$f\circ p_{\Gamma,\Gamma''}=p_{\Gamma,\Gamma'}$$
If $\Gamma'$ is an object over $\Delta$ and  $f:\Gamma\sr \Delta$ is a morphism then let us denote simply by $f^*(\Gamma')$ the object $f^*(\Gamma',n)$ where $n=l(\Gamma')-l(\Delta)$. Note that $n$ can always be inferred from $f$ and $\Gamma'$. 

Similarly we will write simply $q(f,\Gamma')$ for $q(f,\Gamma',n)$ since $n$ can be inferred as $l(\Gamma')-l(codom(f))$. 
\begin{lemma}
\llabel{2015.08.23.l1a}
Let $\Gamma',\Gamma''$ be objects over $\Delta$, $a:\Gamma'\sr \Gamma''$ a morphism over $\Delta$ and $f:\Gamma\sr\Delta$ a morphism. Then there is a unique morphism $f^*(a):f^*(\Gamma')\sr f^*(\Gamma'')$ over $\Gamma$ such that the square
$$
\begin{CD}
f^*(\Gamma') @>q(f,\Gamma')>> \Gamma'\\
@Vf^*(a)VV @VVaV\\
f^*(\Gamma'') @>q(f,\Gamma'')>> \Gamma''
\end{CD}
$$
commutes.
\end{lemma}
\begin{proof}
We have a square
\begin{eq}\llabel{2015.08.23.eq3}
\begin{CD}
f^*(\Gamma'') @>q(f,\Gamma'')>> \Gamma''\\
@Vp_{f^*(\Gamma''),\Gamma}VV @VVp_{\Gamma'',\Delta}V\\
\Gamma @>f>> \Delta\\
\end{CD}
\end{eq}
This square is a pull-back square as a vertical composition of $l(\Gamma'')-l(\Delta)$ pull-back squares. We define $f^*(a)$ as the unique morphism such that 
\begin{eq}\llabel{2015.08.23.eq1}
f^*(a)\circ q(f,\Gamma'')=q(f,\Gamma')\circ a
\end{eq}
and
\begin{eq}\llabel{2015.08.23.eq2}
f^*(a)\circ p_{f^*(\Gamma''),\Gamma}=p_{f^*(\Gamma'),\Gamma}
\end{eq}
The first of these two equalities is equivalent to the commutativity of the square (\ref{2015.08.23.eq3}) and the second to the condition that $f^*(a)$ is a morphism over $\Gamma$.
\end{proof}
\begin{lemma}
\llabel{2015.08.29.l2}
Let $a:\Gamma'\sr\Gamma''$ be a morphism over $\Delta$, $\Gamma'''$ another object over $\Delta$ and suppose that $a$ is a morphism over $\Gamma'''$. Let $f:\Gamma\sr \Delta$ be a morphism. Then one has
\begin{eq}\llabel{2015.08.29.eq2}
f^*(a)=q(f,\Gamma''')^*(a)
\end{eq}
\end{lemma}
\begin{proof}
The morphisms involved in the proof can be seen on the diagram
$$
\begin{CD}
f^*(\Gamma') @>q(f,\Gamma')>> \Gamma'\\
@Vf^*(a)VV @VVaV\\
f^*(\Gamma'') @>q(f,\Gamma'')>> \Gamma''\\
@Vp_{f^*(\Gamma''),f^*(\Gamma''')}VV @VVp_{\Gamma'',\Gamma'''}V\\
f^*(\Gamma''') @>q(f,\Gamma''')>> \Gamma'''\\
@Vp_{f^*(\Gamma'''),\Gamma}VV @VVp_{\Gamma''',\Delta}V\\
\Gamma @>f>> \Delta
\end{CD}
$$
The right hand side of (\ref{2015.08.29.eq2}) is a morphism over $f^*(\Gamma''')$ and therefore a morphism over $\Gamma$. It remains to verify that it satisfies equation (\ref{2015.08.23.eq1}). This follows immediately from its definition. 
\end{proof}
We will also need the following facts about homomorphisms of C-systems. 
\begin{lemma}
\llabel{2015.09.03.l2}
Let $F:CC\sr CC'$ be a homomorphism of C-systems. Then one has:
\begin{enumerate}
\item for $\Gamma\in CC$ and $i\in\nat$ one has $F(p_{\Gamma,i})=p_{F(\Gamma),i}$,
\item for $\Gamma,\Gamma'\in CC$, $\Gamma\le \Gamma'$ implies $F(\Gamma)\le F(\Gamma')$ and similarly for $<$,
\item for $\Gamma'\ge \Delta$ and $f:\Gamma\sr \Delta$ one has
$$F(f^*(\Gamma'))=(F(f))^*(F(\Gamma'))$$
\item for $\Gamma',\Gamma''\ge \Gamma$, $a:\Gamma'\sr \Gamma''$ over $\Delta$ and $f:\Gamma\sr \Delta$ one has
$$F(f^*(a))=(F(f))^*(F(a))$$
\item for $\Gamma$ such that $l(\Gamma)>0$ one has
$$F(\delta(\Gamma))=\delta(F(\Gamma))$$
\end{enumerate}
\end{lemma}
\begin{proof}
The proofs are straightforward and we leave them for the formalized version of the paper.
\end{proof}

\subsection{The presheaf extension of a C-system}
\llabel{Fext}
Let $CC$ be a C-system and $F:CC^{op}\sr Sets$ a presheaf on the category underlying $CC$. In this section we construct a new C-system $CC[F]$ which we call the $F$-extension of $CC$ and describe a unital pre-B-system $B(CC,F)$ and an isomorphism $B(CC[F])\sr B(CC,F)$. 

We will first construct a C0-system $CC[F]$ and then show that it is a C-system. For the definition of a C0-system see \cite[Definition 2.1]{Csubsystems}.
\begin{problem}\llabel{2016.01.19.prob1}
Given a C-system $CC$ and a presheaf $F:CC^{op}\sr Sets$ to construct a C0-system that will be denoted $CC[F]$ and called the $F$-extension of $CC$.
\end{problem}
\begin{construction}\rm\llabel{2016.01.19.constr1}
We set 
\begin{eq}\llabel{2016.01.19.eq1}
Ob(CC[F])=\amalg_{X\in CC} F(ft^{l(X)}(X))\times\dots\times F(ft^2(X))\times F(ft(X))
\end{eq}
where the product of the empty sequence of factors is a 1-point set. We will write elements of $Ob(CC[F])$ as $(X,\Gamma)$ where $X\in CC$ and $\Gamma=(T_0,\dots,T_{l(X)-1})$. Note that $ft^{l(X)}(X)=pt$ for any $X$ and therefore all the products in (\ref{2016.01.19.eq1}) start with $F(pt)$.

We set
$$Mor(CC[F])=\amalg_{(X,\Gamma),(Y,\Gamma')}Mor_{CC}(X,Y)$$
We will write elements of $Mor(CC[F])$ as $((X,\Gamma),(Y,\Gamma'),f)$. When the domain and the codomain of a morphism are clear from the context we may write $f$ instead of $((X,\Gamma),(Y,\Gamma'),f)$. 

We define the composition function by the rule
$$((X,\Gamma),(Y,\Gamma'),f))\circ ((Y,\Gamma'),(Z,\Gamma''),g)=((X,\Gamma),(Z,\Gamma''),f\circ g)$$

We define the identity morphisms by the rule
$$Id_{CC[F],(X,\Gamma)}=((X,\Gamma),(X,\Gamma),Id_{CC,X})$$

The associativity and the identity conditions of a category follow easily from the corresponding properties of  $CC$. This completes the construction of a category  $CC[F]$. 

We define the length function as
$$l((X,\Gamma))=l(X)$$

If $l((X,\Gamma))=0$ then $X=pt$ and $\Gamma=()$ where $()$ is the unique element of the one point set that is the product of the empty sequence. We will often write $(pt,())$ as $pt$. 

We define the ft-function on $(X,\Gamma)$ such that $l(X)>0$ as 
$$ft((X,(T_0,\dots,T_{l(X)-1}))=(ft(X),(T_0,\dots,T_{l(X)-2}))$$
which is well defined because $l(ft(X))=l(X)-1$, and set $ft((pt,()))=(pt,())$. We will write $ft(\Gamma)$ for $(T_0,\dots,T_{l(X)-2})$ so that $ft((X,\Gamma))=(ft(X),ft(\Gamma))$. 

We define the p-morphisms as 
$$p_{(X,\Gamma)}=((X,\Gamma),ft(X,\Gamma), p_X)$$

For $(Y,\Gamma')$ such that $l((Y,\Gamma'))>0$ and $f:(X,\Gamma)\sr ft(Y,\Gamma')$ where $\Gamma=(T_0,\dots,T_{l(X)-1})$ and $\Gamma'=(T_0',\dots,T_{l(Y)-1}')$ we set
\begin{eq}\llabel{2016.01.31.eq1}
f^*((Y,\Gamma'))=(f^*(Y),(T_0,\dots,T_{l(X)-1},F(q(f,Y))(T'_{l(Y)-1}))).
\end{eq}

In the same context as above we define the q-morphism as
$$q(f,(Y,\Gamma'))=(f^*((Y,\Gamma')),(Y,\Gamma'),q(f,Y))$$

This completes the construction of the elements of the structure of a C0-system. Let us verify that these elements satisfy the axioms of a C0-system. 

The uniqueness of an object of length $0$ is obvious.

The condition that $l(ft(X,\Gamma))=l((X,\Gamma))-1$ if $l((X,\Gamma))>0$ is obvious.

The condition that $ft((pt,()))=(pt,())$ is obvious. 

The fact that $pt$ is a final object in $CC[F]$ follows from the fact that $pt$ is a final object of $CC$.

The fact that for $(Y,\Gamma')$ such that $l((Y,\Gamma'))>0$ and $f:(X,\Gamma)\sr ft(Y,\Gamma')$ one has $q(f,(Y,\Gamma'))\circ p_{(Y,\Gamma')}=p_{f^*((Y,\Gamma'))}\circ f$ follows from the corresponding fact in $CC$.

The fact that for $(Y,\Gamma')$ such that $l((Y,\Gamma'))>0$ one has $Id_{ft(Y,\Gamma)}^*((Y,\Gamma'))=(Y,\Gamma')$ follows from the corresponding fact for $CC$ and the identity axiom of the functor $F$. 

The fact that for $(Y,\Gamma')$ such that $l((Y,\Gamma'))>0$ one has $q(Id_{(Y,\Gamma)},(Y,\Gamma))=Id_{(Y,\Gamma)}$ follows from the previous assertion and the corresponding fact in $CC$.

The fact that $(Y,\Gamma')$ such that $l((Y,\Gamma'))>0$, $f:(X,\Gamma)\sr ft(Y,\Gamma')$ and $g:(W,\Delta)\sr (X,\Gamma)$ one has $g^*(f^*((Y,\Gamma')))=(g\circ f)^*((Y,\Gamma'))$ follows from the composition axiom for the functor $F$ and the corresponding fact for $CC$.

The fact that in the same context as in the previous assertion one has
$$q(g,f^*((Y,\Gamma')))\circ q(f,(Y,\Gamma'))=q((g\circ f),(Y,\Gamma'))$$
follows from the previous assertion and the corresponding fact for $CC$. 

This completes Construction \ref{2016.01.19.constr1}
\end{construction}
\begin{lemma}
\llabel{2016.01.19.l2}
The functions $Ob(CC[F])\sr Ob(F)$ and $Mor(CC(F))\sr Mor(CC)$ given by 
$$(X,\Gamma)\mapsto X$$
and
$$((X,\Gamma),(Y,\Gamma'),f)\mapsto f$$
form a functor $tr_F:CC[F]\sr CC$ and this functor is fully faithful.
\end{lemma}
\begin{proof}
Straightforward from the construction.
\end{proof}
\begin{lemma}\llabel{2016.01.19.l1}
The C0-system of Construction \ref{2016.01.19.constr1} is a C-system.
\end{lemma}
\begin{proof}
By \cite[Proposition 2.4]{Csubsystems} it is sufficient to prove that the canonical squares of $CC[F]$, i.e., the squares formed by morphisms $q(f,(Y,\Gamma')),p_{(Y,\Gamma')}$ and $p_{f^*((Y,\Gamma'))}, f$ are pull-back squares. The functor of Lemma \ref{2016.01.19.l2} map these square to canonical squares of the C-system $CC$ that are pull-back squares. Since this functor is fully faithful we conclude that the canonical squares in $CC[F]$ are pull-back squares. The lemma is proved.
\end{proof}

This completes the construction of the presheaf extension of a C-system. 
\begin{remark}\rm
\llabel{2015.09.01.rem1}
For any two objects of $C[F]$ of the form $(X,\Gamma),(X,\Gamma')$ the formula
$$can_{X,\Gamma,\Gamma'}=((X,\Gamma),(X,\Gamma'),Id_X)$$
defines a morphism which is clearly an isomorphism with $can_{X,\Gamma',\Gamma}$ being a canonical inverse. Therefore, all objects of $CC[F]$ with the same image in $CC$ are ``canonically isomorphic''. 
\end{remark}
\begin{remark}\rm
\llabel{2015.09.01.rem2}
If $F(pt)=\emptyset$ then $CC[F]=\{pt\}$. On the other hand, the choice of an element $y$ in $F(pt)$ defines distinguished elements $y_X=F(\pi_X)(y)$ in all sets $F(X)$ and therefore distinguished objects $(X,\Gamma_{X,y})=(X,(y,\dots,y_{ft(X)},y_X))$ in the fibers of the object component of $tr_{F}$ over all $X$. 

Mapping $X$ to $(X,\Gamma_{X,y})$ and $f:X\sr Y$ to $((X,\Gamma_{X,y}),(Y,\Gamma_{Y,y}),f)$ defines, as one can immediately prove from the definitions, a functor $tr_{F,y}^!:CC\sr CC[F]$. 

This functor clearly satisfies the conditions $tr^!_{F,y}\circ tr_F=Id_{CC}$.

One verifies easily that the morphisms 
$$can_{X,\Gamma,\Gamma_{(X,y)}}:(X,\Gamma)\sr tr^!_{F,y}(X,\Gamma)$$
form a natural transformation. We conclude that $tr_F$ and $tr^!_{F,y}$ is a pair of mutually inverse equivalences of categories.

However this equivalence is not an isomorphism unless $F(X)\cong unit$ for all $X$ and as a C-system $CC[F]$ is often very different from $CC$, for example, in that that it may have many more C-subsystems.  
\end{remark}
We provide the following lemma without a proof because the proof is immediate from the definitions and \cite[Lemma 3.4]{Cfromauniverse} that asserts that a functor that satisfies all conditions of the definition of a homomorphism except possibly the s-morphisms condition is a homomorphism. 
\begin{lemma}
\llabel{2015.08.22.l4}
The functor $tr:CC[F]\sr CC$ is a homomorphism of C-systems.
\end{lemma}
\begin{remark}\rm
\llabel{2015.08.22.rem1} 
Let $y\in F(pt)$. Then for $f:X\sr Y$ one has $F(f)(y_{Y})=y_X$ and therefore for $f:X\sr ft(Y)$ one has
$$(tr^!_{y}(f))^*(Y)=(f^*(Y),\Gamma_{f^*(Y),y})=f^*((Y,\Gamma_Y))=tr^!_y(f)^*(tr_y(Y))$$
The rest of the conditions that one needs to prove in order to show that $tr_y$ is a homomorphism of C-systems is immediate from definitions and we obtain that 
$$tr^!_y:CC\sr CC[F]$$
is a homomorphism of C-systems.  
\end{remark}

Recall that by definition $(X,\Gamma)\le (Y,\Gamma')$ if and only if $l(X,\Gamma)\le l(Y,\Gamma')$ and 
$$(X,\Gamma)=ft^{l(Y,\Gamma')-l(X,\Gamma)}(Y,\Gamma').$$
From construction we conclude that $(X,\Gamma)\le (Y,\Gamma')$ if and only if $X\le Y$ in $CC$ and 
$$(X,\Gamma)=ft^{l(Y)-l(X)}((Y,\Gamma')).$$ 
\begin{lemma}
\llabel{2016.01.31.l1}
Let $i\ge 0$, $(Y,\Gamma')$ be such that $l(Y)\ge i$. Let $f:(X,\Gamma)\sr ft^i(Y,\Gamma')$. Let $lx=l(X)$, $ly=l(Y)$ and 
$$\Gamma=(T_0,\dots,T_{lx-1})$$
$$\Gamma'=(T'_0,\dots,T'_{ly-1})$$
Then
$$f^*((Y,\Gamma'),i)=(f^*(Y,i),(T_0,\dots,T_{lx-1},F(q(f,ft^i(Y),0))(T'_{ly-i}),\dots,F(q(f,ft(Y),i-1))(T'_{ly-1}))$$
\end{lemma}
\begin{proof}
By induction on $i$.

For $i=0$ we have 
$$f^*((Y,\Gamma'),0)=(X,\Gamma)=(f^*(Y,0),(T_0,\dots,T_{lx-1}))$$
For the successor of $i$ we need to show that
\begin{eq}\llabel{2016.01.31.eq3}
f^*((Y,\Gamma'),i+1)=$$$$(f^*(Y,i+1),(T_0,\dots,T_{lx-1},F(q(f,ft^{i+1}(Y),0))(T'_{ly-i-1}),\dots,F(q(f,ft(Y),i))(T'_{ly-1}))
\end{eq}
We have by (\ref{2016.01.31.eq2}),  
$$f^*((Y,\Gamma'),i+1)=q(f,ft((Y,\Gamma')),i)^*((Y,\Gamma'))$$
By the inductive assumption, $q(f,ft((Y,\Gamma')),i)$ is a morphism with the domain
$$f^*(ft(Y,\Gamma'),i)=$$$$(f^*(ft(Y),i),(T_0,\dots,T_{lx-1},F(q(f,ft^{i}(ft(Y)),0))(T'_{ly-1-i}),\dots,F(q(f,ft(ft(Y))),i-1)(T'_{ly-2})))$$
By (\ref{2016.01.31.eq1}) we get
$$q(f,ft((Y,\Gamma')),i)^*((Y,\Gamma'))=$$
$$(q(f,ft(Y),i)^*(Y),$$
$$(T_0,\dots,T_{lx-1},F(q(f,ft^{i}(ft(Y)),0))(T'_{ly-1-i}),\dots,F(q(f,ft(ft(Y))),i-1)(T'_{ly-2}), $$$$F(q(f,ft(Y),i))(T'_{ly-1})))$$
which coincides with our goal (\ref{2016.01.31.eq3}). 
\end{proof}
%
%##

\comment{
\begin{lemma}
\llabel{2015.08.26.l7}
Let $(Y,\Gamma')\ge (W,\Delta)$ where $\Gamma'=(T'_0,\dots,T'_{l(Y)-1})$ and let $f:(X,\Gamma)\sr (W,\Delta)$ be a morphism where $\Gamma=(T_0,\dots,T_{l(X)-1})$. Then one has
$$f^*((Y,\Gamma'))=(f^*(Y),(T_0,\dots,T_{l(X)-1},F(q(f,ft^{n'-n}(Y)))(T'_n),\dots,F(q(f,Y))(T'_{n'-1})))$$
\end{lemma}
\begin{proof}
Let $n=l(W)$ and let $i=n'-n$. Then $(W,\Delta)=ft^i((Y,\Gamma'))$, i.e., $W=ft^i(Y)$, $\Delta=(T'_0,\dots,T'_{n-1})$ and $\Gamma'=(T'_0,\dots,T'_{n-1+i})$. By definition $f^*((Y,\Gamma'))=f^*((Y,\Gamma'),i)$. 

The proof is by induction on $i$ after replacing $n'$ with $n+i$. 

For $i=0$ we have $f^*((Y,\Gamma'),0)=(X,\Gamma)$.

For the successor $i+1$ we have, by (\ref{2016.01.31.eq2}),  
$$f^*((Y,\Gamma'),i+1)=q(f,ft((Y,\Gamma')),i)^*((Y,\Gamma'))$$
where, by the inductive assumption, $q(f,ft((Y,\Gamma')),i)$ is a morphism with the domain
$$f^*(ft(Y,\Gamma'),i)=(f^*(ft(Y)),(T_0,\dots,T_{m-1},F(q(f,ft^{i+1}(Y)))(T'_n),\dots,F(q(f,ft(Y)))(T'_{n-1+i})))$$
and by (\ref{2016.01.31.eq1}) we have
$$q(f,ft((Y,\Gamma')),i)^*((Y,\Gamma'))=$$$$(q(f,ft(Y),i)^*(Y),(T_0,\dots,T_{m-1},F(q(f,ft^{i+1}(Y)))(T'_n),\dots,F(q(f,ft(Y)))(T'_{n-2+i}), F(q(q(f,ft((Y)),i),Y))(T'_{n+i}))=$$
$$(q(f,ft(Y),i)^*(Y),(T_0,\dots,T_{m-1},F(q(f,ft^{i+1}(Y)))(T'_n),\dots,F(q(f,ft(Y)))(T'_{n-2+i}), F(q(f,Y,i),i))(T'_{n+i}))$$
\end{proof}
}

\subsection{Some computations with $Jf$-relative monads}
\llabel{Jfrel}

The notion of a relative monad is introduced in \cite[Def.1, p. 299]{ACU} and considered in more detail in \cite{ACU2}. Let us remind it here.
\begin{definition}
\llabel{2015.12.22.def1}
Let $J:C\sr D$ be a functor. A relative monad $\RR$ on $J$ or a $J$-relative monad is a collection of data of the form
\begin{enumerate}
\item a function $RR:Ob(C)\sr Ob(D)$,
\item for each $X$ in $C$ a morphism $\eta(X):J(X)\sr RR(X)$,
\item for each  $X,Y$ in $C$ and $f:J(X)\sr RR(Y)$ a morphism $\mbind(f):RR(X)\sr RR(Y)$,
\end{enumerate}
such that the following conditions hold:
\begin{enumerate}
\item for any $X\in C$, $\mbind(\eta(X))=Id_{RR(X)}$,
\item for any $f:J(X)\sr RR(Y)$, $\eta(X)\circ \mbind(f)=f$,
\item for any $f:J(X)\sr RR(Y)$, $g:J(Y)\sr RR(Z)$, 
$$\mbind(f)\circ \mbind(g)=\mbind(f\circ \mbind(g))$$
\end{enumerate}
\end{definition}
\begin{problem}\llabel{2016.01.15.prob1}
Given a relative monad $\RR$ to construct a functor $(RR_{Ob},RR_{Mor})$ from $C$ to $D$ such that $RR_{Ob}=RR$.
\end{problem}
\begin{construction}\rm\llabel{2016.01.15.constr1}
For $f:X\sr Y$ in $C$ set
$$RR_{Mor}(f)=\mbind(J(f)\circ \eta(Y))$$
The proof of the composition and the identity axioms of a functor are easy.
\end{construction}
For two sets $X$ and $Y$ we let $Fun(X,Y)$ to denote the set of functions from $X$ to $Y$. 

Next, following \cite{FPT} we let $F$ denote the category with the set of objects $\nat$ and the set of morphisms from $m$ to $n$ being $Fun(stn(m),stn(n))$, where $stn(m)=\{i\in\nat\,|\,i<m\}$ is our choice for the standard set with $m$ elements (cf. \cite{LandC}) and where for two sets $X$ and $Y$, 

For any set $U$ there is a category $Sets(U)$ of the following form. The set of objects of $Sets(U)$ is $U$. The set of morphisms is
$$Mor(Sets(U))=\cup_{X,Y\in U}Fun(X,Y)$$
Since a function from $X$ to $Y$ is defined as a triple $(X,Y,G)$ where $G$ is the graph subset of this function the domain and codomain functions are well defined on $Mor(Sets(U))$ such that
$$Mor_{Sets(U)}(X,Y)=Fun(X,Y)$$
and a composition function can be defined that restricts to the composition of functions function on each $Mor_{Sets(U)}(X,Y)$. Finally the identity function $U\sr Mor(Sets(U))$ is obvious and the collection of data that one obtains satisfies the axioms of a category. This category is called the category of sets in $U$ and denoted $Sets(U)$. 

We will only consider the case when $U$ is a universe. As was mentioned in the introduction we fix $U$ and omit it from our notations below. 

Following \cite{ACU} we let $Jf:F\sr Sets$ denote the functor that takes $n$ to $stn(n)$ and that is the identity on morphisms between two objects (on the total sets of morphisms the morphism component of this functor is the inclusion of a subset). 

As the following construction shows any monad on sets defines a $Jf$-relative monad. Combined with our construction of $C(\RR)$ this gives a construction of a C-system for any monad on sets. 

\begin{problem}\llabel{2016.01.13.prob1}
Given a monad ${\bf R}=(R,\eta,\mu)$ (cf. \cite{MacLane}[p. 133]) on the category of sets to construct a $Jf$-relative monad $\RR$.
\end{problem}
\begin{construction}\rm\llabel{2016.01.13.constr1}
We set 
\begin{enumerate}
\item $R(n)=R(stn(n))$,
\item $\eta_n=\eta_{stn(n)}$,
\item for $f:stn(m)\sr R(n)$ we set $\mbind(f)=R(f)\circ \mu_{stn(n)}$.
\end{enumerate}
The verification of the relative monad axioms is easy.
\end{construction}
\begin{remark}\rm\llabel{2016.01.03.rem1}
It seems to be possible to provide a construction of a monad from a $Jf$-relative monad without the use of the axioms of choice and excluded middle. This construction will be considered in a separate note.
\end{remark}
\begin{remark}\rm\llabel{2016.01.17.rem1}
The set of $Jf$-relative monads is in an easy to construct bijection with the set of abstract clones as defined in \cite[Section 3]{FPT}. 
\end{remark}

In \cite{LandJf} we constructed for any $Jf$-relative monad $\RR=(RR,\eta,\mbind)$ a Lawvere theory $(T,L)=RML(\RR)$. Most of this section is occupied by simple computations in $T$ that will be used in the later sections.

Recall that the category $T$ has as the set of objects the set of natural numbers and as the set of morphisms the set 
$$Mot_T=\amalg_{m,n} Fun(stn(m),RR(n))$$
Therefore the set of morphisms in $T$ from $m$ to $n$ is the set of iterated pairs $((m,n),f)$ where $f\in Fun(stn(m),RR(n))$. We fix the obvious bijection between this set and $Fun(stn(m),RR(n))$ and use the corresponding functions in both directions as coercions. A coercion, in the terminology of the proof assistant Coq, is a function $f:X\sr Y$ such that when an expression denoting an element $x$ of the set $X$ occurs in a position where an element of $Y$ should be it is assumed that $x$ is replaced by $f(x)$.

Let us introduce the following notation:
$$F(m,n)=Fun(stn(m),stn(n))$$
and, for a $Jf$-relative monad $\RR$,
$$RR(m,n)=Fun(stn(m),RR(n))$$
Then for $f\in RR(l,m)$ and $g\in RR(m,n)$ the composition $f\hc g$ in $T$ is defined as $\mbind(f)\circ g$ and for $m\in\nat$ the identity morphism $Id_m$ in $T$ is defined as $\eta(m)$. 

The functor $L:F\sr T$ is defined as the identity on objects and as the function on morphisms corresponding to the functions $f\mapsto f\circ \eta(n)$ from $F(m,n)$ to $RR(m,n)$. 

We also obtain the extension of $RR$ to a functor $F\sr Sets$ according to Construction \ref{2016.01.15.constr1}. For a morphism $f\in F(m,n)$ we have $RR(f)=\mbind(f\circ \eta(n))=\mbind(L(f))$. 

We are going to use the functions $f\mapsto RR(f)$ as coercions so that when an element $f$ of $F(m,n)$ occurs in a position where an element of $Fun(RR(m),RR(n))$ is expected it has to be replaced by $RR(f)$. 

\comment{Similarly, we will use the functions $\mbind(m)$ as a coercions so that when an element $g$ of $RR(m,n)$ occurs in a position where a function from $RR(m)$ to $RR(n)$ is expected it has to be replaced by $\mbind(g)$. }
\begin{remark}\rm
\llabel{2015.11.20.rem4}
We can not replace $\amalg$ by $\cup$ in our definition of the set of morphisms of $T$ because for a general $\RR$ the sets $RR(m,n)$ are not disjoint.  For example, if $RR(m)=pt$ where $pt$ is a fixed one element set then $RR$ has a (unique) structure of a $Jf$-relative monad and $RR(m,n)=RR(m,n')$ for all $m,n,n'$. Therefore no function to $\nat$ from the union of these sets can distinguish the codomain of a morphism. In particular, in this case there is no category with the sets of morphisms from $m$ to $n$ being equal $RR(m,n)$. 
\end{remark}
Since we will have to deal with elements of the sets of functions $Fun(stn(m),RR(n))$ and of similar sets such as the sets $Ob_n(C(\RR,\LM))$ introduced later we need to choose some way to represent them. For the purpose of the present paper we will write such elements as sequences, i.e., to denote the function, which in the notation of $\lambda$-calculus is written as $\lambda\,i:stn(n), f_i$, we will write $(f_0,\dots,f_{n-1})$. In particular, for an element $x$ of a set $X$, the expression $(x)$ denotes the function $stn(1)\sr X$ that takes $0$ to $x$. 
\begin{lemma}
\llabel{2016.01.15.l4}
Let $f=(f(0),\dots,f(l-1))$ be a morphism in $T$ from $l$ to $m$ and $g=(g(0),\dots,g(m-1))$ a morphism from $m$ to $n$. Then one has
$$f\hc g=(\mbind(g)(f(0)),\dots,\mbind(g)(f(l-1)))$$
\end{lemma}
\begin{proof}
We have
$$(f\hc g)(i)=(f\circ\mbind(g))(i)=\mbind(g)(f(i)).$$
The lemma is proved. 
\end{proof}
\begin{lemma}
\llabel{2015.08.30.l1}
Let $f\in F(l,m)$, $g\in RR(m,n)$ and $i\in stn(l)$. Then one has
\begin{eq}\llabel{2015.08.26.eq4}
(L(f)\hc g)(i)=g(f(i))
\end{eq}
\end{lemma}
\begin{proof}
Rewriting the left hand side we get 
$$(L(f)\hc g)(i)=((f\circ \eta(m))\circ \mbind(g))(i)=(f\circ (\eta(m)\circ \mbind(g)))(i)=(f\circ g)(i)=g(f(i)).$$
which completes the proof. 
\end{proof}
For $n\in\nat$ and $i=0,\dots,n-1$ let
$$x_i^n=\eta(n)(i)\in RR(n)$$
Observe also that for $f\in RR(m,n)$ one has
\begin{eq}\llabel{2015.08.24.eq5}
\mbind(f)(x_i^m)=(\eta(m)\circ\mbind(f))(i)=f(i)
\end{eq}
and for $f\in F(m,n)$ one has
\begin{eq}\llabel{2016.01.15.eq1}
f(x_i^m)=RR(f)(\eta(m)(i))=(\eta(m)\circ \mbind(f\circ \eta(n)))(i)=(f\circ \eta(n))(i)=\eta(n)(f(i))=x_{f(i)}^n
\end{eq}
Let 
$$\partial^{i}_{n}:stn(n)\sr stn(n+1)$$
for $0\le i\le n$ be the increasing inclusion that does not take the value $i$ and
$$\sigma^{i}_{n}:stn(n+2)\sr stn(n+1)$$
for $0\le i\le n$ be the non-decreasing surjection that takes the value $i$ twice. Taking into account that, in the notation of \cite{GabZis}, $[n]=stn(n+1)$ these are the standard generators of the simplicial category $\Delta$ together with $\partial^0_0:stn(0)\sr stn(1)$. 

In our sequence notation we have
\begin{eq}\llabel{2015.08.24.eq7}
L(\partial^{i}_{n})=(x_0^{n+1},\dots,x_{i-1}^{n+1},x_{i+1}^{n+1},\dots,x_n^{n+1})
\end{eq}
and
\begin{eq}\llabel{2015.08.24.eq8}
L(\sigma^{i}_{n})=(x_0^{n+1},\dots,x_{i-1}^{n+1},x_i^{n+1},x_i^{n+1},x_{i+1}^{n+1},\dots,x_n^{n+1})
\end{eq}
in particular
\begin{eq}\llabel{2015.07.12.eq5}
L(\partial^{n}_n)=(x_0^{n+1},\dots,x_{n-1}^{n+1})
\end{eq}

Let 
$${\iota}_n^{i}:stn(n)\sr stn(n+i)$$
be the function given by ${\iota}_n^i(j)=j$ for $j=0,\dots,n-1$. Then we have
\begin{eq}
\llabel{2015.08.22.eq7}
{\iota}_n^{1}=\partial_{n}^{n}
\end{eq}
and (\ref{2016.01.15.eq1}) implies that 
\begin{eq}
\llabel{2015.08.22.eq8}
{\iota}_n^i(x^n_j)=x^{n+i}_j
\end{eq}
\begin{lemma}
\llabel{2015.08.26.l1}
Let $f=(f(0),\dots,f(m))$ be a morphism from $m+1$ to $n$ in $T$. Then
\begin{eq}\llabel{2016.01.15.eq3}
L(\iota_m^1)\hc f=(f(0),\dots,f(m-1))
\end{eq}
In particular, if $f\in RR(n+1,n)$ then $L(\iota_n^1)\hc f=Id_{T,n}$ if and only if $f(i)=x_n^i$ for $i=0,\dots,n-1$. 
\end{lemma}
\begin{proof}
Both sides of the required equality are elements of $Fun(stn(m),RR(n))$. Therefore, the equality holds if and only if for all $i=0,\dots,n-1$ we have
$(L(\iota_m^1)\hc f)(i)=f(i)$. The assertion of the lemma follows now from Lemma \ref{2015.08.30.l1}.

Since $Id_{T,n}=(x_0^n,\dots,x_{n-1}^n)$ the second assertion immediately follows from the first one.
\end{proof}
For $f\in RR(n,m)$, $f=(f(0),\dots,f(n-1))$ define an element $qq(f)\in RR(n+1,m+1)$ by the formula: 
\begin{eq}\llabel{2015.08.26.eq9}
qq(f)=(\iota_m^1(f(0)),\dots,\iota_m^1(f(n-1)),x_m^{m+1})
\end{eq}
\begin{lemma}
\llabel{2015.08.26.l2}
For $i\in\nat$ and $f=(f(0),\dots,f(n-1))$ in $RR(n,m)$ one has
$$qq^i(f)=(\iota_m^i(f(0)),\dots,\iota_m^i(f(n-1)),x_m^{m+i},\dots,x_{m+i-1}^{m+i})$$
\end{lemma}
\begin{proof}
Straightforward by induction on $i$.
\end{proof}
\begin{lemma}
\llabel{2015.08.26.l3a}
For $n,i\in\nat$ one has
$$qq^i(L(\iota_n^1))=L(\partial^n_{n+i})$$
\end{lemma}
\begin{proof}
We have ${L}(\iota_n^1)=L(\partial_n^n)=(x_0^{n+1},\dots,x^{n+1}_{n-1})$. By Lemma \ref{2015.08.26.l2} and (\ref{2015.08.22.eq8}) we get
$$qq^i({L}(\iota_n^1))=(\iota_{n+1}^i(x_0^{n+1}),\dots,\iota_{n+1}^i(x_{n-1}^{n+1}),x_{n+1}^{n+1+i},\dots,x_{n+i}^{n+1+i})=$$
$$=(x_0^{n+1+i},\dots,x_{n-1}^{n+1+i},x_{n+1}^{n+1+i},\dots,x_{n+i}^{n+1+i})={L}(\partial_n^{n+i})$$
where the last equality is (\ref{2015.08.24.eq7}). 
\end{proof}
\begin{lemma}
\llabel{2015.08.28.l1}
For $i,m\in\nat$ and $r\in RR(m)$ one has
$$qq^i(x_0^m,\dots,x^m_{m-1},r)=(x_0^{m+i},\dots,x^{m+i}_{m-1},\iota_m^i(r),x_m^{m+i},\dots,x_{m+i-1}^{m+i})$$
\end{lemma}
\begin{proof}
One has
$$qq^i(x_0^m,\dots,x^m_{m-1},r)=(\iota_m^i(x_0^m),\dots,\iota_m^i(x^m_{m-1}),\iota_m^i(r),x_m^{m+i},\dots,x_{m+i-1}^{m+i})=$$
$$(x_0^{m+i},\dots,x^{m+i}_{m-1},\iota_m^i(r),x_m^{m+i},\dots,x_{m+i-1}^{m+i})$$
where the first equality is by Lemma \ref{2015.08.26.l2} and the second one by (\ref{2015.08.22.eq8}). 
\end{proof}

\subsection{The C-system $C(\RR)$}
\llabel{CRR}

In \cite{LandC} we constructed for any Lawvere theory $(T,L)$ a C-system $LC((T,L))$. For $(T,L)=RML(\RR)$ we denote the C-system $LC((T,L))$ by $C(\RR)$. In this section we first provide a more explicit description of $C(\RR)$ and then compute the action of the operations $T,\wt{T},S,\wt{S}$ and $\delta$ on the B-sets $(Ob(CC(\RR)),\wt{Ob}(CC(\RR)))$ of this C-system (cf. Definition \ref{2015.08.26.def1}).    

Recall that as a category $C(\RR)$ is the opposite category to $T$.  To distinguish the positions in formulas where natural numbers are used as objects of $C(\RR)$ we will write in such places $\wh{m}$ instead of $m$, $\wh{n}$ instead of $n$ etc.  

We consider $L$ as a functor
$${L}:F^{op}\sr C(\RR)$$
i.e., as a contravariant functor from $F$ to $C(\RR)$ and keep the conventions introduced in the previous section the most important of which is that for $f\in F(m,n)$ and $x\in RR(m)$ we write $f(x)$ for $RR(f)(x)=\mbind(f\circ \eta(n))(x)$. 

The $ft$ function on $C(\RR)$ is defined by the formula $ft(\wh{n+1})=\wh{n}$ and $ft(\wh{0})=\wh{0}$.

The $p$-morphisms are defined by setting $p_{\wh{0}}=Id_{\wh{0}}$ and $p_{\wh{n+1}}:\wh{n+1}\sr \wh{n}$ to be the morphism $L(\iota_n^1)$. In the sequence notation we have
\begin{eq}\llabel{2015.08.24.eq6}
p_{\wh{n+1}}=(x_0^{n+1},\dots,x_{n-1}^{n+1})
\end{eq}
For a morphism $f:\wh{m}\sr \wh{n}$ in $C(\RR)$ we have $f^*(\wh{n+1})=\wh{m+1}$. 

Before giving an explicit description of $q$-morphisms we will prove the following lemma.
\begin{lemma}
\llabel{2015.07.24.l1}
One has:
\begin{enumerate}
\item Let $f=(f(0),\dots,f(n))$ be a morphism $\wh{m+1}\sr\wh{n+1}$. Then 
$$f\circ_C p_{\wh{n+1}}=(f(0),\dots,f(n-1))$$
\item Let $f=(f(0),\dots,f(n-1))$ be a morphism $\wh{m}\sr \wh{n}$. Then 
$$p_{\wh{m+1}}\circ_C f=(\iota_m^1(f(0)),\dots,\iota_m^1(f(n-1)))$$
\end{enumerate}
\end{lemma}
\begin{proof}
Both sides of the first equality are elements of $Fun(stn(n), RR(m+1))$ and for $i\in stn(n)$ we have
$$(f\circ_C p_{\wh{n+1}})(i)=(L(\iota_{n}^1)\circ_T f)(i)=f(i)$$
where the second equality is by (\ref{2015.08.26.eq4}). 

Both sides of the second equality are again elements of $Fun(stn(n), RR(m+1))$ and for $i\in stn(n)$ we have:  
$$(p_{\wh{m+1}}\circ_C f)(i)=(f\circ_T L(\iota_{m}^1))(i)=(f\circ \mbind(L(\iota_{m}^1)))(i)=(f\circ RR(\iota_{m}^1))(i)=\iota_{m}^1(f(i))$$
\end{proof}
The $q$-morphisms were defined in \cite{LandC} in a somewhat implicit manner. We give their explicit description in the following lemma.
\begin{lemma}
\llabel{2016.01.15.l3}
Let $f:\wh{m}\sr \wh{n}$ be a morphism in $C(\RR)$. Then one has
$$q(f,\wh{n+1})=qq(f)$$
\end{lemma}
\begin{proof}
The morphism $q(f)=q(f,\wh{n+1})$ was defined in \cite{LandC} as the unique morphism such that
$$q(f)\circ_C p_{\wh{n+1}}=p_{\wh{m+1}}\circ_C f$$
and
$$q(f)\circ_C (x_n^{n+1})=(x_m^{m+1})$$
For the first equation we have
$$qq(f)\circ_C p_{\wh{n+1}}=(\iota_m^1(f(0)),\dots,\iota_m^1(f(n-1)))$$
by Lemma \ref{2015.07.24.l1}(1) and (\ref{2015.08.26.eq9}) and
$$p_{\wh{m+1}}\circ_C f = (\iota_m^1(f(0)),\dots,\iota_m^1(f(n-1)))$$
by Lemma \ref{2015.07.24.l1}(2).

Both sides of the second equation are elements of $Fun(stn(1),RR(m+1))$ and it is sufficient that their values on $0$ coincide. We have 
$$(q(f)\circ_C (x_n^{n+1}))(0)=((x_n^{n+1})\circ_T qq(f))(0)=((x_n^{n+1})\circ \mbind(qq(f)))=$$$$\mbind(qq(f))(x_n^{n+1})=qq(f)(n)=x_m^{m+1}$$
where the fourth equality is by (\ref{2015.08.24.eq5}) and the fifth by (\ref{2015.08.26.eq9}). This completes the proof of Lemma \ref{2016.01.15.l3}.
\end{proof}

Let us describe the constructions introduced in Section \ref{onCsystems} in the case of $C(\RR)$. Note that our wide-hat notation that distinguishes the places in formulas where natural numbers are used as objects of $C(\RR)$ allows us to avoid the ambiguity that might have arisen otherwise. For example $p_{m,n}$ could be understood either as the canonical morphism $m\sr n$ using the notation $p_{\Gamma',\Gamma}$ introduced in Section \ref{onCsystems} or as the canonical morphism $m\sr m-n$ using the notation $p_{\Gamma,i}$ that we have used in \cite{Csubsystems}.  The use of the wide-hat diacritic allows to distinguish between $p_{\wh{m},\wh{n}}$ - a morphism $\wh{m}\sr \wh{n}$,  and $p_{\wh{m},n}$ - a morphism $\wh{m}\sr \wh{m-n}$. 
\begin{lemma}
\llabel{2015.08.22.l6}
Let $n,i\in\nat$. 
\begin{enumerate}
\item One has
\begin{enumerate}
\item $p_{\wh{n+i},i}=L(\iota^{n+i}_{n})=(x_0^{n+i},\dots,x_{n-1}^{n+i})$,
\item for $m\in\nat$ and $g=(g(0),\dots,g(n+i-1))$ from $\wh{m}$ to $\wh{n+i}$ one has 
$$g\circ p_{\wh{n+i},i}=(g(0),\dots,g(n-1)),$$
\end{enumerate}
\item for $f:\wh{m}\sr \wh{n}$ one has
$$f^*(\wh{n+i},i)=m+i$$
and
$$q(f,\wh{n+i},i)=qq^i(f)$$
\end{enumerate}
\end{lemma}
\begin{proof}
All three assertions a proved by induction on $i$. For the first assertion both parts are proved by induction simultaneously. One has
\begin{enumerate}
\item in the case $i=0$ the first assertion follows from the identity axiom of the functor defined by $\RR$ as in Construction \ref{2016.01.15.constr1} and second from the identity axiom of the category $C(\RR)$,
\item for the successor of $i$ we have
$$p_{\wh{n+i+1},i+1}=p_{\wh{n+i+1}}\circ p_{\wh{n+i},i}=(x_0^{n+i+1},\dots,x_{n-1}^{n+i+1})$$
where the second equality is by the second part of the inductive assumption. For the inductive step in the second part we have
$$(g(0),\dots,g(n+i))\circ p_{\wh{n+i+1},i+1}=(g(0),\dots,g(n+i))\circ p_{\wh{n+i+1}}\circ p_{\wh{n+i},i}=$$$$
(g(0),\dots,g(n+i-1))\circ p_{\wh{n+i},i}=(g(0),\dots,g(n-1))$$
\end{enumerate}
The proof of the first part of the second assertion is obvious. For the second part we have:
\begin{enumerate}
\item for $i=0$ the assertion is obvious,
\item for the successor of $i$ we have 
$$q(f,\wh{n+i+1},i+1)=qq(q(f,\wh{n+i},i))=qq(qq^i(f))=qq^{i+1}(f)$$
\end{enumerate}
\end{proof}
\begin{lemma}
\llabel{2015.08.22.l7}
Let $f=(f(0),\dots,f(n))$ be a morphism from $\wh{n}$ to $\wh{n+1}$. Then $f\circ p_{\wh{n+1}}=Id_{\wh{n}}$ if and only if $f(i)=x_i^{n}$ for $i=0,\dots,n-1$.
\end{lemma}
\begin{proof}
It follows immediately from Lemma \ref{2015.08.26.l1}.
\end{proof}
\begin{lemma}
\llabel{2015.09.09.l1}
Let $f=(f(0),\dots,f(n-1))$ be a morphism from $\wh{m}$ to $\wh{n}$ where $n>0$. Then one has
$$s_f=(x_0^{m},\dots,x_{m-1}^{m},f(n-1))$$
\end{lemma}
\begin{proof}
By \cite[Definition 2.3(2)]{Csubsystems} we have that 
$$s_f\circ p_{\wh{m+1}}=Id_{\wh{m}}$$
Therefore, by Lemma \ref{2015.08.22.l7}, $s_f$ is of the form $(x_0^m,\dots,x_{m-1}^m,sf)$ for some $sf\in RR(m)$. By \cite[Definition 2.3(3)]{Csubsystems} we have $f=s_f\circ q(ft(f),\wh{n})$ where $ft(f)=f\circ p_{\wh{n}}$. By Lemma \ref{2015.07.24.l1}(1) we have $ft(f)=(f(0),\dots,f(n-2))$ and by  Lemma \ref{2016.01.15.l3} and (\ref{2015.08.26.eq9}) we have
$$q(ft(f),\wh{n})=qq(ft(f))=(\iota_m^1(f(0)),\dots,\iota_m^1(f(n-2)),x_{m}^{m+1})$$
Therefore we should have
$$(f(0),\dots,f(n-1))=(\iota_m^1(f(0)),\dots,\iota_m^1(f(n-2)),x_{m}^{m+1})\hc (x_0^m,\dots,x_{m-1}^m,sf)$$
which is equivalent to, by Lemma \ref{2016.01.15.l4}, 
\begin{eq}\llabel{2016.01.15.eq6}
f(i)=\mbind(x_0^m,\dots,x_{m-1}^m,sf)(\iota_m^1(f(i)))
\end{eq}
for $i=0,\dots,n-2$ and 
\begin{eq}\llabel{2016.01.15.eq7}
f(n-1)=\mbind(x_0^m,\dots,x_{m-1}^m,sf)(x_{m}^{m+1})
\end{eq}
For the first series of equalities we get, by inserting the coercion $RR$ and rewriting of the right hand side, the following
$$\mbind(x_0^m,\dots,x_{m-1}^m,sf)(\iota_m^1(f(i)))=(\mbind(L(\iota_m^1))\circ \mbind(x_0^m,\dots,x_{m-1}^m,sf))(f(i))=$$$$\mbind(L(\iota_m^1)\circ \mbind(x_0^m,\dots,x_{m-1}^m,sf))(f(i))=\mbind(L(\iota_m^1)\hc (x_0^m,\dots,x_{m-1}^m,sf))(f(i))=$$$$\mbind((x_0^m,\dots,x_{m-1}^m))(f(i))=\mbind(\eta(m))(f(i))=Id_{RR(m)}(f(i))=f(i)$$
where the fourth equality is by (\ref{2016.01.15.eq3}). 

Equality (\ref{2016.01.15.eq7}) gives us, by (\ref{2015.08.24.eq5}) that $sf=f(n-1)$.
\end{proof}

Recall from \cite{Csubsystems} that for a C-system $CC$ one defines $\wt{Ob}(CC)$ as the subset of $Mor(CC)$ which consists of morphisms $s$ of the form $ft(X)\sr X$ such that $l(X)>0$ and $s\circ p_X=Id_{ft(X)}$. 
\begin{lemma}
\llabel{2015.08.24.l1}
Let $f:\wh{m}\sr\wh{n}$ and let $s:\wh{n}\sr\wh{n+1}$ be an element of $\wt{Ob}$. Then one has
$$f^*(s)=(x_0^m,\dots,x_{m-1}^m,\mbind(f)(s(n)))$$
\end{lemma}
\begin{proof}
The fact that the first $m$ terms of the sequence representation of $fs=f^*(s)$ have the required form follows from Lemma \ref{2015.08.22.l7}. It remains to prove that
$$fs(m)=\mbind(f)(s(n))=(s\hc f)(n)$$
The morphism $f^*(s)$, as a morphism over $\wh{m}$ is defined by the equation
$$f^*(s)\circ_C q(f,\wh{n+1})=f\circ_C s$$
which is equivalent, by Lemma \ref{2016.01.15.l3},  to $qq(f)\hc fs=s\hc f$. Therefore 
$$(s\hc f)(n)=(qq(f)\hc fs)(n)=\mbind(fs)(qq(f)(n))=\mbind(fs)(x_m^{m+1})=\mbind(fs)(\eta(m+1)(m))=$$$$(\eta(m+1)\circ\mbind(fs))(m)=fs(m).$$ 
The lemma is proved.
\end{proof}
\begin{lemma}
\llabel{2015.08.29.l1}
Let $f:\wh{m}\sr\wh{n}$ and let $s:\wh{n+i}\sr\wh{n+i+1}$ be an element of $\wt{Ob}$. Then one has
\begin{eq}\llabel{2015.08.29.eq1}
f^*(s)=(x_0^{m+i},\dots,x_{m+i-1}^{m+i},\mbind(qq^i(f))(s(n+i)))
\end{eq}
\end{lemma}
\begin{proof}
The morphisms involved in the proof can be seen on the following diagram
$$
\begin{CD}
\wh{m+i} @>qq^i(f)>> \wh{n+i}\\
@Vf^*(s)VV @VVsV\\
\wh{m+i+1} @>qq^{i+1}(f)>> \wh{n+i+1}\\
@Vp_{m+i+1,i+1}VV @VVp_{n+i+1,i+1}V\\
\wh{m} @>f>> \wh{n}
\end{CD}
$$
The morphism $s$ is a morphism from $Id_{\wh{n+i}}$ to $p_{\wh{n+i+1}}$ over $\wh{n+i}$. Therefore, we may apply Lemma \ref{2015.08.29.l2} obtaining the equality
$$f^*(s)=(qq^i(f))^*(s)$$
On the other hand by Lemma \ref{2015.08.24.l1} we have
$$qq^i(f)^*(s)=(x_0^{m+i},\dots,x_{m+i-1}^{m+i},\mbind(qq^i(f))(s(n+i))).$$
The lemma is proved. 
\end{proof}
Another operation that we would like to have an explicit form of is operation $\delta$. For a C-system $CC$ and an object $\Gamma$ in $CC$ such that $l(\Gamma)>0$ one defines $\delta_{\Gamma}$ as $s_{Id(\Gamma)}$ (cf. \cite[Section 3]{Csubsystems}). 
\begin{lemma}\llabel{2015.08.24.l5}
In $C(\RR)$ one has: 
$$\delta_{\wh{n}}=(x_0^{n},\dots,x_{n-1}^{n},x_{n-1}^n)$$
\end{lemma}
\begin{proof}
It follows from Lemma \ref{2015.09.09.l1} since $Id_{\wh{n}}=(x_0^{n},\dots,x_{n-1}^{n})$. 
\end{proof}
\begin{problem}
To construct a bijection
\begin{eq}\llabel{2015.08.24.eq9}
mb_{\RR}:\wt{Ob}(C(\RR))\sr \amalg_{n\in\nat} RR(n)
\end{eq}
\end{problem}
\begin{construction}\rm
\llabel{2015.08.22.constr3}
For $s:\wh{n}\sr \wh{n+1}$ define
$$mb_{\RR}(s)=(n,s(n))$$
To show that this is a bijection let us construct the inverse bijection. For $n\in\nn$ and $o\in RR(n)$ set
$$mb_{\RR}^!(n,o)=(x_0^n,\dots,x_{n-1}^n,o)$$
The fact that these functions are mutually inverse follows easily from Lemma \ref{2015.08.22.l7}.
\end{construction}

Our next goal is to describe operations $T'$, $\wt{T}'$, $S'$, $\wt{S}'$ and $\delta'$ obtained from operations $T$, $\wt{T}$, $S$, $\wt{S}$ and $\delta$ that were introduced at the end of Section 3 in \cite{Csubsystems} through transport by means of the bijection (\ref{2015.08.24.eq9}).

Let us first recall the definition of operations $T$, $\wt{T}$, $S$, $\wt{S}$ and $\delta$ associated with a general C-system $CC$. 
\begin{definition}
\llabel{2015.08.26.def1}
Let $CC$ be a C-system. We will write $Ob$ for $Ob(CC)$ and $\wt{Ob}$ for $\wt{Ob}(CC)$. 
\begin{enumerate}
\item Operation $T$ is defined on the set
$$T_{dom}=\{\Gamma,\Gamma'\in Ob\,|\,l(\Gamma)>0\,\,and\,\, \Gamma'>ft(\Gamma)\}$$
and takes values in $Ob$. For $(\Gamma,\Gamma')\in T_{dom}$ one defines
$$T(\Gamma,\Gamma')=p_{\Gamma}^*(\Gamma')$$
\item Operation $\wt{T}$ is defined on the set
$$\wt{T}_{dom}=\{\Gamma\in Ob, s\in \wt{Ob}\,|\,l(\Gamma)>0\,\,and\,\, \partial(s)>ft(\Gamma)\}$$
and takes values in $\wt{Ob}$. For $(\Gamma,s)\in \wt{T}_{dom}$ one defines
$$\wt{T}(\Gamma,s)=p_{\Gamma}^*(s)$$
\item Operation $S$ is defined on the set
$$S_{dom}=\{r\in \wt{Ob}, \Gamma\in Ob\,|\,\Gamma>\partial(r)\}$$
and takes values in $Ob$. For $(r,\Gamma)\in S_{dom}$ one defines
$$S(r,\Gamma)=r^*(\Gamma)$$
\item Operation $\wt{S}$ is defined on the set 
$$\wt{S}_{dom}=\{r,s\in \wt{Ob}\,|\,\partial(s)>\partial(r)\}$$
and takes values in $\wt{Ob}$. For $(r,s)\in \wt{S}_{dom}$ one defines
$$S(r,s)=r^*(s)$$
\item Operation $\delta$ is defined on the set 
$$\delta_{dom}=\{\Gamma\in Ob\,|\,l(\Gamma)>0\}$$
and takes values in $\wt{Ob}$. For $\Gamma\in \delta_{dom}$ one defines $\delta(\Gamma)$ as $s_{Id_{\Gamma}}$. 
\end{enumerate}
\end{definition}

Define, for any $Jf$-relative monad $\RR$ operations $\theta_{m,n}=\theta^{\RR}_{m,n}$ such that for $m,n\in\nat$, $n>m$ and $r\in RR(m)$, $s\in RR(n)$ one has
\begin{eq}
\llabel{2015.09.07.eq1}
\theta_{m,n}(r,s)=\mbind(qq^{n-m-1}(x_0^m,\dots,x_{m-1}^m,r))(s)=$$
$$\mbind(x_0^{n-1},\dots,x_{m-1}^{n-1},\iota_{m}^{n-m-1}(r),x_m^{n-1},\dots,x_{n-2}^{n-1})(s)
\end{eq}
\begin{theorem}
\llabel{2015.08.26.th1}
Let $Ob=Ob(C(\RR))$ and let $\wt{Ob}'$ be the right hand side of (\ref{2015.08.24.eq9}). One has:
\begin{enumerate}
\item Operation $T'$ is defined on the set
$$T'_{dom}=\{\wh{m},\wh{n}\in Ob\,|\,m>0\,\,and\,\,n>m-1\}$$
and is given by 
$$T'(\wh{m},\wh{n})=\wh{n+1}$$
\item Operation $\wt{T}'$ is defined on the set 
$$\wt{T}'_{dom}=\{\wh{m}\in Ob, (n,s)\in \wt{Ob}'\,|\,m>0\,\,and\,\,n+1>m-1\}$$
and is given by
$$\wt{T}'(\wh{m},(n,s))=(n+1,\partial_n^{m-1}(s))$$
\item Operation $S'$ is defined on the set
$$S'_{dom}=\{(m,r)\in \wt{Ob}',\wh{n}\in Ob\,|\,n>m+1\}$$
and is given by
$$S'((m,r),\wh{n})=\wh{n-1}$$
\item Operation $\wt{S}'$ is defined on the set 
$$\wt{S}'_{dom}=\{(m,r)\in\wt{Ob}',(n,s)\in \wt{Ob}'\,|\,n>m\}$$
and is given by
$$\wt{S}'((m,r),(n,s))=\theta_{m,n}(r,s)$$
\item Operation $\delta'$ is defined on the subset
$$\delta'_{dom}=\{\wh{n}\in Ob\,|\,n>0\}$$
and is given by
$$\delta'(\wh{n})=(n,x_{n-1}^n)$$
\end{enumerate}
\end{theorem}
\begin{proof}
We have:
\begin{enumerate}
\item Operation $T'$ is the same as operation $T$ for $C(\RR)$ since $\wt{Ob}$ is not involved in it. The form of $T'_{dom}$ is obtained by unfolding definitions and the formula for the operation itself follows from Lemma \ref{2015.08.22.l6}(2).
\item Operation $\wt{T}'$ is defined on the set of pairs $(\wh{m}\in Ob, (n,s)\in \wt{Ob}')$ such that $m>0$ and $\partial(mb_{\RR}^!(n,s))>m-1$. Since $\partial(mb_{\RR}^!(n,s))=n+1$ we obtain the required domain of definition. The formula by the operation itself is obtained immediately by combining Lemma \ref{2015.08.29.l1} and Lemma \ref{2015.08.26.l3a}.
\item Operation $S'$ is defined on the set of pairs $((m,r)\in \wt{Ob}',\wh{n}\in Ob)$ where $n>\partial(mb_{\RR}^!(m,r))$. Since $\partial(mb_{\RR}^!(m,r))=m+1$ we obtain the required domain of definition. The operation itself is given by
$$S'((m,r),n)=(mb_{\RR}^!(m,r))^*(\wh{n})=(x_0^m,\dots,x_{m-1}^m,r)^*(\wh{n})=\wh{n+m-(m+1)}=\wh{n-1}$$
\item Operation $\wt{S}'$ is defined on the set of pairs $(m,r),(n,s)\in \wt{Ob}'$ 
such that $\partial(mb_{\RR}^!(n,s))>\partial(mb_{\RR}^!(m,r))$ which is equivalent to $n>m$. The formula for the operation itself is obtained immediately by combining Lemma \ref{2015.08.29.l1} with $i=n-m-1$ and  Lemma \ref{2015.08.28.l1}. 
\item Operation $\delta'$ is defined on the subset $\wh{n}\in Ob$ such that $n>0$ and is given by
$$\delta'(\wh{n})=mb_{\RR}(\delta(\wh{n}))=mb_{\RR}((x_0^n,\dots,x_{n-1}^n,x_{n-1}^n))=(n,x_{n-1}^n)$$
\end{enumerate}
The theorem is proved. 
\end{proof}
The length function on $Ob=\nat$ is the identity. Of  the remaining three operations that define the pre-B-system structure on the pair of sets $(Ob,\wt{Ob}')$ - $pt$, $ft$ and $\partial'$, the first two are described above and $\partial'$ is given by $\partial'((m,r))=m+1$. 

This completes the description of the pre-B-system structure on $(Ob,\wt{Ob}')$ that is obtained by the transport of structure from the standard pre-B-system structure on $(Ob,\wt{Ob})$ by means of the pair of isomorphisms $Id$ and $mb_{\RR}$.
\begin{remark}\rm
\llabel{2015.08.29.rem2}
Conjecturally, a C-system can be reconstructed (up to an isomorphism) from the sets $Ob$ and $\wt{Ob}$ equipped with the length function $l:Ob\sr\nn$, the distinguished object $pt\in Ob$ and operations $ft, \partial, T,\wt{T},S,\wt{S}$ and $\delta$. Combining this conjecture with Theorem \ref{2015.08.26.th1} we conclude that the C-system $C(\RR)$ and, therefore, the relative monad $\RR$, can be reconstructed from the sets $RR(n)$ with distinguished elements $x^n_i$ and equipped with operations $\partial_n^i$ and $\theta_{m,n}:RR(m)\times RR(n)\sr RR(n-1)$ for $n>m$. 

Using Remark \ref{2016.01.17.rem1} this can be compared with the assertion of \cite[Theorem 3.3]{FPT} that the category of abstract clones is equivalent to the category of substitution systems of \cite[Definition 3.1]{FPT}. In such a comparison the operation $\zeta$ of substitution systems of the form $RR(n+1)\times RR(n)\sr RR(n)$ is the same as the operation $(s,r)\mapsto \theta_{n,n+1}(r,s)$. 
\end{remark}
\begin{remark}\rm
\llabel{2015.08.29.rem1}
Let $lRR$ be the disjoint union of $RR(n)$ for all $n$. Then we can sum up all of the operations that we need to consider as follows: 
\begin{enumerate}
\item a function $l:lRR\sr\nat$,
\item a function $\eta:\nat\sr lRR$ that takes $n$ to $x^n_0=\eta(n)(0)$,
\item a function $\partial:\{r\in lRR, i\in\nat\,|\, l(r)\ge i\}\sr lRR$,
\item a function $\theta:\{r,s\in lRR,\,|\,l(r)>l(s)\}\sr lRR$,
\end{enumerate}
such that
\begin{enumerate}
\item for all $n\in\nat$, $l(\eta(n))=n+1$,
\item for all $r\in lRR$, $i\in\nat $ such that $l(r)\ge i$, $l(\partial(r,i))=l(r)+1$,
\item for all $r,s\in lRR$ such that $l(s)>l(r)$ one has $l(\sigma(r,s))=l(s)-1$.
\end{enumerate}
It should be possible to describe, by a collection of further axioms on these operations, a full subcategory in the category whose objects are sets $lRR$ with operations of the form $l,\eta,\partial$ and $\theta$ that is equivalent to the category of $Jf$-relative monads or, equivalently, the category of Lawvere theories or Fiore-Plotkin-Turi substitution algebras. 
\end{remark}
\begin{remark}\rm
\llabel{2015.08.29.rem1b}
It seems at first unclear why it should be possible to realize the action of the symmetric group on $RR(n)$ using operations of Remark \ref{2015.08.29.rem2} since they all seem to respect, in some sense, the linear ordering of the sets $stn(n)$. 

In the substitution notation of Remark \ref{2015.08.18.rem1}, given $r$ in $RR(m)$ and $E$ in $RR(n)$, 
$$\theta_{m,n}(r,E)=E[r/x_m,x_m/x_{m+1},\dots,x_{n-2}/x_{n-1}],$$
i.e., the operation $\theta_{m,n}$ corresponds to the substitution of an expression in variables $x_0,\dots,x_{m-1}$ for the variable $x_m$ in an expression in variables $x_0,\dots,x_n$ followed by a downshift of the indexes of the variables with the higher index. 

The operation $\partial_n^i$ and the constants $x_n:=x^{n+1}_{n}$ are similarly defined in terms of linear orderings.

To see how it is, nevertheless, possible to realize, for example, the permutation of $x_0$ and $x_1$ consider the following. First let, for all $i,n\in\nat$, 
$$\iota_n^i=\partial_{n+i-1}^{n+i-1}\circ \dots\circ \partial_n^n:RR(n)\sr RR(n+i)$$
Then define for all $i,n\in\nat$, $n\ge i+1$ an element $x^n_i\in RR(n)$ by the formula
$$x^n_i=\iota_{i+1}^{n-i-1}(x_i)$$
such that, in particular, $x^{n+1}_{n}=x_n$.

Define now a function $\psi:RR(2)\sr RR(2)$ by the formula
$$\psi=\partial^0_2\circ \partial^0_3\circ \theta_{3,4}(x_0^3,-)\circ \theta_{2,3}(x_1^2,-)$$
One can verify that for any $Jf$-relative monad $RR$, $\psi=\sigma$ where $\sigma$ is the permutation of $0$ and $1$ in $stn(2)$. 

In the substitution notation this can be seen as follows:
$$\psi(E(x^2_0,x_1^2))=\theta_{2,3}(x_1^2,\theta_{3,4}(x_0^3,\partial^0_3(\partial^0_2(E(x^2_0,x_1^2)))))=\theta_{2,3}(x_1^2,\theta_{3,4}(x_0^3,\partial^0_3(E(x^3_1,x^3_2))))=$$
$$\theta_{2,3}(x_1^2,\theta_{3,4}(x_0^3,E(x^4_2,x^4_3)))=\theta_{2,3}(x_1^2,E(x^3_2,x^3_0))=E(x^2_1,x^2_0)$$
\end{remark}

\subsection{The C-system $C(\RR,\LM)$.}
\llabel{CRRLM}

Modules (actually left modules) over relative monads were introduced in \cite[Definition 9]{Ahrens2016}. One can observe by direct comparison of unfolded definitions that there is a bijection between the set of modules over a relative monad $\RR$ with values in a category $E$ and the set of functors from the Kleisli category $K(\RR)$ of $\RR$ introduced in \cite[p.8]{ACU2} (see also \cite[Constr. 2.9]{LandJf}) to $E$. Whether this bijection is the identity bijection or not depends on how the expressions such as ``collection of data'' or ``family of functions'' are translated into the formal constructions of set theory. We assume that they have been translated in a such a way that this bijection is the identity and left modules over $\RR$ with values in $E$ are actually and precisely the same as (covariant) functors from $K(\RR)$ to $E$. 

In this paper we are interested in the $Jf$-relative monads $\RR$. The corresponding Kleisli categories are the categories opposite to the categories $C(\RR)$ underlying the C-systems considered above. Therefore, left modules over a $Jf$-monad $\RR$ with values in $Sets$ are the presheaves on $C(RR)$, i.e., the contravariant functors from $C(RR)$ to $Sets$. 

Let $\LM=(LM,LM_{Mor})$ be such a presheaf.

The morphism component $LM_{Mor}$ of $\LM$ is a function that sends a morphism $f$ from $\wh{m}$ to $\wh{n}$ in $C(\RR)$ to a function $LM_{Mor}(f)\in Fun(LM(\wh{n}),LM(\wh{m}))$, i.e., we have for each $m,n\in\nat$ a function
$$R(n,m)\sr Fun(LM(\wh{n}),LM(\wh{m}))$$
We will use this function as a coercion so that, for $f\in RR(n,m)$ and $E\in LM(\wh{n})$ the expression $f(E)$ is assumed to be expanded into $LM_{Mor}(f)(E)$ when needed. 
\begin{remark}\rm
\llabel{2015.08.18.rem1}
If we think of $E\in LM(\wh{n})$ as of an expression in variables $0,\dots,n-1$ then the action of $RR(n,m)$ on $LM(\wh{n})$ can be thought of as the substitution. This analogy can be used to introduce the notation when for $f=(f(0),\dots,f(n-1))\in RR(n,m)$ and $E\in LM(\wh{n})$ one writes $f(E)$ as 
$$f(E)=E[f(0)/0,\dots,f(n-1)/n-1]$$
For example, in this notation we have 
$$\partial^i_n(E)=E[0/0,\dots,i-1/i-1,i+1/i,\dots,n/n-1]$$
Similarly, for $E\in LM(\wh{n+2})$ one has
$$\sigma^i_n(E)=E[0/0,\dots,i/i,i/i+1,\dots,n/n+1]$$
and $\iota_n^i(E)$ is ``the same expression'' but considered as an expression of $n+i$ variables.
\end{remark}
\begin{example}\rm
\llabel{2015.09.07.rem3}
An important example of $\LM$ is given by the functor defined on objects by $\wh{n}\mapsto RR(n)$ and on morphisms by 
$$f\mapsto (s\mapsto\mbind(f)(s))$$
for $f:\wh{m}\sr \wh{n}$ and $s\in RR(n)$. We will denote this functor by the same symbol $\RR$ as the underlying $Jf$-relative monad. 

This functor is isomorphic to the (contravariant) functor represented by the object $\wh{1}$ but it is not equal to this functor since the set of elements of the form  $((\wh{n},\wh{1}),r')$ where $r'\in RR(1,n)$ is isomorphic but not equal to the set $RR(n)$.
\end{example}

Let $C(\RR,\LM)=C(\RR)[\LM]$ be the $\LM$-extension of the C-system $C(\RR)$. The role of these C-systems in the theory of type theories is that the term C-systems of the raw syntax of dependent type theories are of this form and therefore the term C-systems of dependent type theories are regular sub-quotients of such C-systems and can be studied using the description of the regular sub-quotients given in \cite{Csubsystems}.  

By construction,
\begin{eq}\llabel{2016.01.21.eq3}
Ob(C(\RR,\LM))=\amalg_{n\in\nat}Ob_n(\RR,\LM)
\end{eq}
where 
$$Ob_n(\RR,\LM)=\LM(\wh{0})\times\dots\times\LM(\wh{n-1})$$
and therefore objects of $C(\RR,\LM)$ are pairs of the form $(n,\Gamma)$ where $\Gamma$ is a sequence $(T_0,\dots,T_{n-1})$ where $T_i\in LM(\wh{i})$. While the number $n$ in a pair $(n,\Gamma)$ is an object of $C(\RR)$ we will not add the ${\,\,\wh{}\,\,}$ diacritic to it since no confusion of the kind possible with objects of $C(RR)$ and objects of $F$ can arise. We may sometimes omit $n$ from our notation altogether since it can be recovered from $\Gamma$. Similarly, while the morphisms of $C(\RR,\LM)$ are given by iterated pairs of the form $(((m,\Gamma),(n,\Gamma')),((\wh{m},\wh{n}),f))$ where $f\in \RR(n,m)$ we will sometimes write them as $f:(m,\Gamma)\sr (n,\Gamma')$ or $f:\Gamma\sr \Gamma'$ or even just as $f$. 

Let us also recall that for two objects $X=(m,(T_0,\dots,T_{m-1})))$ and $Y=(n+1,(T'_0,\dots,T'_{n}))$ and a morphism $f:X\sr ft(Y)$ the object $f^*(Y)$ is given by the formula
\begin{eq}\llabel{2015.09.09.eq3old}
f^*(Y)=(m+1,(T_0,\dots,T_{m-1},f(T'_{n})))
\end{eq}
and the morphism $q(f,Y):f^*(Y)\sr Y$ by the formula $q(f,Y)=qq(f)$. 
\begin{lemma}
\llabel{2015.08.26.l8}
Let $X=(m,(T_0,\dots,T_{m-1}))$ and $Y=(n,(T_0,\dots,T_{n-2},T))$ where $m>n-1$. Then one has
$$p_{Y}^*(X)=(m+1,(T_0,\dots,T_{n-2},T,\partial_{n-1}^{n-1}(T_{n-1}),\dots,\partial_{m-1}^{n-1}(T_{m-1})))$$
\end{lemma}
\begin{proof}
We want to apply Lemma \ref{2016.01.31.l1}. We have $lx=n$, $ly=m$. The morphism $p_Y$ is of the form
$$p_Y=p_{\wh{n}}:(n,(T_0,\dots,T_{n-2},T))\sr (n-1,(T_0,\dots,T_{n-2}))$$
and 
$$(n-1,(T_0,\dots,T_{n-2}))=ft^i((m,(T_0,\dots,T_{m-1})))$$
where $i=m-n+1$. Therefore, 
$$p_Y^*(X)=p_Y^*(X,i)=$$$$(p_{\wh{n}}^*(\wh{m},i), (T_0,\dots,T_{n-2},T,q(p_{\wh{n}},ft^i(\wh{m}),0)(T_{m-i}),\dots,q(p_{\wh{n}},ft(\wh{m}),i-1)(T_{m-1}))=$$
$$(m+1,(T_0,\dots,T_{n-2},T,\iota_{n-1}^1(T_{n-1}),\dots,qq^{i-1}(L(\iota_{n-1}^1))(T_{m-1})))=$$
$$(m+1,(T_0,\dots,T_{n-2},T,\partial_{n-1}^{n-1}(T_{n-1}),\dots,\partial_{m-1}^{n-1}(T_{m-1})))$$
where the third equality is by Lemma \ref{2015.08.22.l6}(2) and the fourth one by Lemma \ref{2015.08.26.l3a}. 
\end{proof}
\begin{lemma}
\llabel{2015.08.22.l5}
A morphism $f:X\sr Y$, where $l(Y)=n+1$ and $f\in R(n+1,n)$ belongs to $\wt{Ob}(C(\RR,\LM))$ if and only if $X=ft(Y)$ and $f(i)=x^n_i$ for $i=0,\dots,n-1$. 
\end{lemma}
\begin{proof}
It follows immediately from Lemma \ref{2015.08.26.l1}.
\end{proof}
The following analog of Lemma \ref{2015.09.09.l1} for the C-system $C(\RR,\LM)$ provides us with the explicit form of the operation $f\mapsto s_f$.
\begin{lemma}
\llabel{2015.09.09.l2}
Let $f:X\sr Y$, $f=(f(0),\dots,f(n-1))$ where $n>0$. Then $s_f:X\sr (ft(f))^*(Y)$,
\begin{eq}\llabel{2015.09.09.eq1}
s_f=(x_0^{m},\dots,x_{m-1}^{m},f(n-1))
\end{eq}
where $ft(f)=f\circ p_Y$ and $m=l(X)$. 
\end{lemma}
\begin{proof}
By definition $s_f$ is a morphism from $X$ to $(ft(f))^*(Y)$. Therefore it is sufficient to show that the left hand side of (\ref{2015.09.09.eq1}) agrees with the right hand side after application of the homomorphism $tr_{\LM}$ and our goal follows from Lemma \ref{2015.09.09.l1}.
\end{proof}
\begin{lemma}
\llabel{2015.09.03.l1}
For $i>0$, $f:X\sr ft^i(Y)$ and $s:ft(Y)\sr Y$ in $\wt{Ob}(C(\RR,\LM))$ one has $s:f^*(ft(Y))\sr f^*(Y)$,
$$f^*(s)=(x_0^{m+i-1},\dots,x_{m+i-2}^{m+i-1},\mbind(qq^{i-1}(f))(s(n+i-1)))$$
where $m=l(\Gamma')$ and $n=l(\Gamma)$.
\end{lemma}
\begin{proof}
Since $tr_{\LM}$ is fully faithful, it is sufficient, in order to verify the equality of two morphisms to verify that their domain and codomain are equal and that their images under $tr_{\LM}$ are equal. For the domain and codomain it follows from the definition of $f^*$ on morphisms. For the images under $tr_{\LM}$ it follows from the fact that $tr_{\LM}$ is a homomorphism of C-systems, Lemma \ref{2015.09.03.l2}(4) and
Lemma \ref{2015.08.29.l1}. 
\end{proof}
\begin{problem}
\llabel{2015.08.22.prob1}
To construct a bijection
\begin{eq}
\llabel{2009.10.15.eq2}
mb_{\RR,\LM}:\wt{Ob}(C(\RR,\LM))\sr \coprod_{n\in\nat} Ob_{n+1}(\RR,\LM)\times R(n)
\end{eq}
\end{problem}
\begin{construction}\rm
\llabel{2015.08.22.constr1}
\llabel{2014.06.30.l2}
Let $s\in \wt{Ob}(C(\RR,\LM))$. Then $s:ft(X)\sr X$, $s\in R(n,n+1)$ and $X=(n+1,\Gamma)$. We set:
$$mb_{\RR,\LM}(s)=(n,(\Gamma, s(n)))$$
To show that this is a bijection let us construct an inverse. For $n\in\nat$, $\Gamma\in Ob_{n+1}(\RR,\LM)$ and $o\in R(n)$ let
$$mb_{\RR,\LM}^!(n,(\Gamma,o))=((ft((n+1,\Gamma)),(n+1,\Gamma)),(x^{n}_0,\dots,x^{n}_{n-1},o))$$
This is a morphism from $ft(X)$ to $X$ where $X=(n+1,\Gamma)$. The equation $mb_{\RR,\LM}^!(n,(\Gamma,o))\circ p_X=Id_{ft(X)}$ follows from Lemma \ref{2015.08.22.l5}. 

Let us show now that $mb_{\RR,\LM}$ and $mb_{\RR,\LM}^!$ are mutually inverse bijections. Let $s\in \wt{Ob}$ be as above, then:
$$mb_{\RR,\LM}^!(mb_{\RR,\LM}(s))=mb_{\RR,\LM}^!(n,(\Gamma,s(n)))=((ft(X),X),(x^{n}_0,\dots,x^{n}_{n-1},s(n)))=s$$
where the last equality follows from the assumption that $s\in \wt{Ob}$ and Lemma \ref{2015.08.22.l5}.

On the other hand for $\Gamma\in Ob_{n+1}(\RR,\LM)$ and $o\in R(n)$ we have
$$mb_{\RR,\LM}(mb_{\RR,\LM}^!(n,(\Gamma,o)))=mb_{\RR,\LM}(ft((n+1,\Gamma)),((n+1,\Gamma),(x^{n}_0,\dots,x^{n}_{n-1},o)))=$$
$$(n,(\Gamma,o))$$
This completes Construction \ref{2015.08.22.constr1}.
\end{construction}
\begin{lemma}
\llabel{2015.09.09.l3}
Let $f:X\sr Y$, $f=(f(0),\dots,f(n-1))$ where $X=(m,(T_0,\dots,T_{m-1}))$, $Y=(n,(T_0',\dots,T_{n-1}'))$. Then one has
$$mb_{\RR,\LM}(s_f)=(m,((T_0,\dots,T_{m-1},(f(0),\dots,f(n-2))(T_{n-1}')), f(n-1)))$$
\end{lemma}
\begin{proof}
It follows immediately from Lemma \ref{2015.09.09.l2} and the formula for $mb_{\RR,\LM}$.
\end{proof}
Consider operations $T'$, $\wt{T}'$, $S'$, $\wt{S}'$ and $\delta'$ obtained by transport by means of the bijection of Construction \ref{2015.08.22.constr1} from the operations $T$, $\wt{T}$, $S$ and $\wt{S}$ and $\delta$ corresponding to the C-system $C(\RR,\LM)$ (cf. Definition \ref{2015.08.26.def1}).  Let us give an explicit description of these operations. 

Recall that we defined, for any $Jf$-relative monad $\RR$, operations 
$$\theta^{\RR}_{m,n}:RR(m)\times RR(n)\sr RR(n-1)$$
For $\LM$ as above and $n>m$ define operations $\theta^{\LM}_{m,n}$ of the form
$$\theta^{\LM}_{m,n}:RR(m)\times LM(n)\sr LM(n-1)$$
by the formula
\begin{eq}\llabel{2015.09.07.eq2}
\theta^{\LM}_{m,n}(r,E)=$$$$(qq^{n-m-1}(x_0^m,\dots,x_{m-1}^m,r))(E)=(x_0^{n-1},\dots,x_{m-1}^{n-1},\iota_m^{n-m-1}(r),x_m^{n-1},\dots, x_{n-2}^{n-1})(E)
\end{eq}
where the second equality is the equality of Lemma \ref{2015.08.28.l1}. As in the case of $\theta^{\RR}_{m,n}$ we will often write $\theta_{m,n}$ instead of $\theta^{\LM}_{m,n}$ since the whether we consider $\theta^{\RR}$ or $\theta^{\LM}$ can be inferred from the type of the arguments.

\begin{theorem}\llabel{2015.08.26.th2}
Let $Ob=Ob(C(\RR,\LM))$ and let $\wt{Ob}'=\wt{Ob}'(\RR,\LM)$ be the right hand side of (\ref{2009.10.15.eq2}). One has:
\begin{enumerate}
\item Operation $T'$ is defined on the set $T'_{dom}$ of pairs $(m,\Gamma),(n,\Gamma')\in Ob$ where $\Gamma=(T_0,\dots,T_{m-1})$, $\Gamma'=(T_0',\dots,T_{n-1}')$ such that $m>0$, $n>m-1$ and $T_i=T_i'$ for $i=0,\dots,m-2$. It takes values in $Ob$ and is given by 
$$T((m,\Gamma),(n,\Gamma'))=$$$$(n+1,(T_0,\dots,T_{m-2},T_{m-1},\partial_{m-1}^{m-1}(T'_{m-1}),\dots,\partial_{n-1}^{m-1}(T'_{n-1})))$$
\item Operation $\wt{T}'$ is defined on the set $\wt{T}'_{dom}$ of pairs $(m,\Gamma)\in Ob$, $(n,(\Gamma',s))\in\wt{Ob}'$ where $\Gamma=(T_0,\dots,T_{m-1})$, $\Gamma'=(T_0',\dots,T_{n-1}')$ such that $m>0$, $n+1>m-1$ and $T_i=T_i'$ for $i=0,\dots,m-2$. It takes values in $\wt{Ob'}$ and is given by
$$\wt{T}'((m,\Gamma),(n,(\Gamma',s)))=(n+1,(T((m,\Gamma),(n,\Gamma')),\partial_{n}^{m-1}(s)))$$
\item Operation $S'$ is defined on the set of pairs $(m,(\Gamma,r))\in \wt{Ob}'$, $(n,\Gamma')\in Ob$ where $\Gamma=(T_0,\dots,T_{m})$, $\Gamma'=(T_0',\dots,T_{n-1}')$ such that $n>m+1$ and $T_i=T_i'$ for $i=0,\dots,m$. It takes values in the set $Ob$ and is given by 
$$S'((m,(\Gamma,r)),(n,\Gamma'))=$$$$(n-1,(T_0',\dots,T_{m-1}',\theta_{m,m+1}(r,T_{m+1}'),\theta_{m,m+2}(r,T_{m+2}'),\dots,\theta_{m,n-1}(r,T_{n-1}')))$$
\item Operation $\wt{S}'$ is defined on the set of pairs $(m,(\Gamma,r))\in \wt{Ob}'$, $(n,(\Gamma',s))\in \wt{Ob}'$ where $\Gamma=(T_0,\dots,T_{m})$, $\Gamma'=(T_0',\dots,T_{n}')$ such that $n>m$ and $T_i=T_i'$ for $i=0,\dots,m$. It takes values in $\wt{Ob}'$ and is given by 
$$\wt{S}'((m,(\Gamma,r)),(n,(\Gamma',s)))=(n-1,(S'((m,(\Gamma,r)),(n+1,\Gamma'))),\theta_{m,n}(r,s))$$
\item Operation $\delta'$ is defined on the subset of $(m,\Gamma)$ in $Ob$ such that $m>0$. It takes values in $\wt{Ob}'$ and is given by 
$$\delta'((m,\Gamma))=(m,(T((m,\Gamma),(m,\Gamma)),x_{m-1}^m))$$
\end{enumerate}
\end{theorem}
\begin{proof}
In the proof we will write $mb$ and $mb^!$ instead of $mb_{\RR,\LM}$ and $mb^!_{\RR,\LM}$. We have:
\begin{enumerate}
\item Operation $T'$ is the same as operation $T$ for $C(\RR,\LM)$ since $\wt{Ob}$ is not involved in it. The form of $T'_{dom}$ is obtained by unfolding definitions. 

The operation itself is given by 
$$T'((m,\Gamma),(n,\Gamma'))=p_{(m,\Gamma)}^*((n,\Gamma'))=$$
$$(m,(T_0,\dots,T_{m-1},\partial_{m-1}^{m-1}(T'_{m-1}),\dots, \partial_{n-1}^{m-1}(T_{n-1}')))$$
where the first equality is by Definition \ref{2015.08.26.def1}(1) and the second by Lemma \ref{2015.08.26.l8}. 
\item Operation $\wt{T}'$ is defined on the set of pairs $(m,\Gamma)\in Ob$, $(n,(\Gamma',s))\in \wt{Ob}'$ such that $m>0$ and $\partial(mb^!(n,(\Gamma',s)))>ft(m,\Gamma)$ and takes values in $\wt{Ob}'$. Since $\partial(mb^!(n,(\Gamma',s))=(n+1,\Gamma')$ we obtain the required domain by unfolding definitions. 

To verify the formula for the operation itself consider the equalities:
$$\wt{T}'((m,\Gamma),(n,(\Gamma',s)))=mb(p_{(m,\Gamma)}^*(mb^!(n,(\Gamma',s))))=$$
$$mb(p_{(m,\Gamma)}^*((ft((n+1,\Gamma')),((n+1,\Gamma'),(x^{n}_0,\dots,x^{n}_{n-1},s)))))$$
where the first equality is by Definition \ref{2015.08.26.def1}(2). By Lemma \ref{2015.09.03.l1} we can extend these equalities as follows:
$$mb(p_{(m,\Gamma)}^*((ft((n+1,\Gamma')),((n+1,\Gamma'),(x^{n}_0,\dots,x^{n}_{n-1},s)))))=$$
$$mb(p_{X}^*(ft(Y)),(p_{X}^*(Y),(x^{n+1}_0,\dots,x^{n+1}_{n},(qq^{n-m+1}(\iota^1_{m-1}))(s))))=$$
$$(n+1,(p_{X}^*(Y),\partial_{n}^{m-1}(s)))=(n+1,(T((m,\Gamma),(n+1,\Gamma')),\partial_n^{m-1}(s)))$$
where $X=(m,\Gamma)$, $Y=(n+1,\Gamma')$, the first equality is by Lemma \ref{2015.09.03.l1}, the second by Lemma \ref{2015.08.26.l3a} and the third by Definition \ref{2015.08.26.def1}(1). 
\item Operation $S'$ is defined on the set of pairs $((m,(\Gamma,r))\in \wt{Ob}',(n,\Gamma')\in Ob)$ such that $(n,\Gamma')>\partial(mb^!(m,(\Gamma,r)))$ and takes values in $Ob$. Since $\partial(mb^!(m,(\Gamma,r)))=(m+1,\Gamma)$ we obtained the required domain of definition. The operation itself is given by
\begin{eq}\llabel{2016.01.21.eq2}
S'((m,(\Gamma,r)),(n,\Gamma'))=(mb^!((m,(\Gamma,r))))^*((n,\Gamma'))
\end{eq}
Next we have 
$$(mb^!((m,(\Gamma,r))))^*((n,\Gamma'))=$$$$((ft(A),A),(x^{m}_0,\dots,x^{m}_{m-1},r))^*(B)=((ft(A),A),(x^{m}_0,\dots,x^{m}_{m-1},r))^*(B,i)$$
%##
where $A=(m+1,\Gamma)$, $B=(n,\Gamma')$ and $i=n-m-1$. To apply Lemma \ref{2016.01.31.l1} we should take $X=\wh{m}$, $lx=m$ and $Y=\wh{n}$, $ly=n$ and $f=((ft(A),A),(x^{m}_0,\dots,x^{m}_{m-1},r))$. Let further $rr=(x^{m}_0,\dots,x^{m}_{m-1},r)$. Then we can extend these equalities as follows
$$f^*((n,\Gamma'),i)=(rr^*(\wh{n},i),(T_0,\dots,T_{m-1},q(rr,ft^i(\wh{n}),0)(T'_{m+1}),\dots,q(rr,ft(\wh{n}),i-1)(T'_{n-1})))$$
$$(n-1,(T_0,\dots,T_{m-1},rr(T'_{m+1}),\dots,qq^{n-m-2}(rr)(T'_{n-1})))=$$
$$(n-1,(T_0',\dots,T_{m-1}',rr(T_{m+1}'),qq(rr)(T_{m+2}'),\dots,qq^{n-m-2}(rr)(T_{n-1}')))$$
where the last equality holds by the assumption that $T_i=T_i'$ for $i=0,\dots,m$. 
The required formula follows from the equality 
$$qq^j(rr)(T_{m+j+1}')=\theta_{m,m+j+1}(r,T_{m+j+1}').$$
\item Operation $\wt{S}'$ is defined on the set of pairs $(m,(\Gamma,r))\in \wt{Ob}'$, $(n,(\Gamma',s))\in \wt{Ob}'$ such that 
\begin{eq}\llabel{2016.01.21.eq1}
\partial(mb^!((n,(\Gamma',s))))>\partial(mb^!(m,(\Gamma,r)))
\end{eq}
and takes values in $\wt{Ob}'$. The inequality (\ref{2016.01.21.eq1}) is equivalent to 
$$(n+1,\Gamma')>(m+1,\Gamma)$$
which is, in turn, equivalent to the conditions in the theorem. In the computation  below let us sometimes abbreviate $((X,Y),f)$ to $f$. Let
$$rr=(x_0^m,\dots,x_{m-1}^m,r)$$
$$ss=(x_0^n,\dots,x_{n-1}^n,s)$$
Then the operation itself is given by:
$$\wt{S}'((m,(\Gamma,r)),(n,(\Gamma',s)))=mb((mb^!(m,(\Gamma,r)))^*(mb^!((n,(\Gamma',s)))))=mb(rr^*ss)=$$$$mb((x_0^{n-1},\dots,x^{n-1}_{n-2},(qq^{n-m-1}(rr))(s)))=(n-1,(rr^*((n+1,\Gamma')),(qq^{n-m-1}(rr))(s)))=$$
$$(n-1,(S'((m,(\Gamma,r)),(n+1,\Gamma'))),\theta_{m,n}(r,s))$$
where the third equality is by Lemma \ref{2015.09.03.l1} and the fifth by (\ref{2016.01.21.eq2}) and the definition of $\theta_{m,n}(r,s)$. 
\item Operation $\delta'$ is defined on the subset $(m,\Gamma)\in Ob$ such that $m>0$ and is given by
$$\delta'((m,\Gamma))=mb(\delta((m,\Gamma)))$$
Therefore it is sufficient to show that
$$\delta((m,\Gamma))=(((m,\Gamma),p_{(m,\Gamma)}^*((m,\Gamma))),(x_0^m,\dots,x_{m-1}^m,x_{m-1}^m))$$
By Definition \ref{2015.08.26.def1}(5), $\delta((m,\Gamma))$ is a morphism from $(m,\Gamma)$ to $p_{(m,\Gamma)}^*((m,\Gamma))$. Therefore, since $tr_{\LM}$ is a fully faithful functor it is sufficient to show that 
$$tr_{\LM}(\delta((m,\Gamma)))=((\wh{m},\wh{m+1}),(x_0^m,\dots,x_{m-1}^m,x_{m-1}^m))$$
which follows from Lemma \ref{2015.09.03.l2}(5) and Lemma \ref{2015.08.24.l5}.
\end{enumerate}
This completes the proof of the theorem. 
\end{proof}
The length function on $Ob$ is described above. Of the remaining three operations that define the pre-B-system structure on the pair of sets $(Ob,\wt{Ob}')$ - $pt$, $ft$ and $\partial'$, the first two are described above as well and $\partial'$ is given by $\partial'((m,(\Gamma,r)))=(m+1,\Gamma)$. 

This completes the description of the pre-B-system structure on $(Ob,\wt{Ob}')$ that is obtained by the transport of structure from the standard pre-B-system structure on $(Ob,\wt{Ob})$ by means of the pair of isomorphisms $Id$ and $mb_{\RR,\LM}$.
\begin{remark}\rm
\llabel{2015.09.13.rem1}
Given an $Jf$-relative monad $\RR$ in the form  $l\RR=(lR,l,\eta,\partial,\theta)$ of Remark \ref{2015.08.29.rem1} we can define a left l-module $l\LM$ over $\RR$ as a quadruple:
\begin{enumerate}
\item a set $lLM$,
\item a function $l:lLM\sr \nn$,
\item a function $\partial:\{E\in lLM, i\in\nat\,|\, l_{\LM}(E)\ge i\}\sr lLM$,
\item a function $\theta^{\LM}:\{r\in lR,E\in lLM\,|\,l_{\LM}(E)>l_{\RR}(r)\}\sr lLM$
\end{enumerate}
where operations $l,\partial$ and $\theta^{\LM}$ satisfy some conditions. 

Once these conditions are properly established the category of such pairs $(l\RR,l\LM)$ should be equivalent to the Hirschowitz-Maggesi ``large module category'' category (see \cite[Definition 2.9]{HM2008})  and in particular the systems of expressions associated with binding signatures can be described as universal objects carrying some additional operations in this category. 

These l-versions of the relative monads and their modules should be easier to formalize in systems such as HOL.  
\end{remark} 

{\em Acknowledgements:} 
\begin{enumerate}
\item Work on this paper was supported by NSF grant 1100938.
\item This material is based on research sponsored by The United States Air Force Research Laboratory under agreement number FA9550-15-1-0053. The US Government is authorized to reproduce and distribute reprints for Governmental purposes notwithstanding  any copyright notation thereon.

The views and conclusions contained herein are those of the author and should not be interpreted as necessarily representing the official policies or endorsements, either expressed or implied, of the United States Air Force Research Laboratory, the U.S. Government or Carnegie Mellon University.
\end{enumerate}

%## Functoriality

\def\cprime{$'$}

%\bibliography{../../../alggeom}
%\bibliographystyle{plain}

\end{document}